\documentclass[a4paper,10pt]{amsart}
\usepackage[colorlinks,linkcolor=blue,citecolor=blue]{hyperref}
\usepackage{latexsym, amssymb, amsmath, amsthm, bbm}
\usepackage[all]{xy}

\usepackage{anysize}\marginsize{22mm}{22mm}{30mm}{30mm}

\allowdisplaybreaks[4]

\def \To{\longrightarrow}
\def \dim{\operatorname{dim}}

\def \mod{\operatorname{mod}}
\def \o{\otimes}
\def \b{\Delta}
\def \C{\mathcal{C}}

\def \D{\Delta}
\def \d{\delta}
\def \e{\varepsilon}

\def \S{\mathcal{S}}

\def \Z{\mathbbm{Z}}

\def \k{\mathbbm{k}}
\def \1{\mathbf{1}}
\def \id{\operatorname{id}}

\numberwithin{equation}{section}

\newtheorem{theorem}{Theorem}[section]
\newtheorem{lemma}[theorem]{Lemma}
\newtheorem{proposition}[theorem]{Proposition}
\newtheorem{corollary}[theorem]{Corollary}
\newtheorem{definition}[theorem]{Definition}

\newtheorem{remark}[theorem]{Remark}

\begin{document}

\title[Quasi-Quantum Linear Spaces]{Quasi-Quantum Linear Spaces$^\dag$}\thanks{$^\dag$Supported by SRFDP 20130131110001, SDNSF ZR2013AM022 and NSFC 11571199.}



\author[H.-L. Huang]{Hua-Lin Huang}
\address{School of Mathematics, Shandong University, Jinan 250100, China} \email{hualin@sdu.edu.cn}

\author[Y. Yang]{Yuping Yang*}\thanks{*Corresponding author.}
\address{School of Mathematics, Shandong University, Jinan 250100, China} \email{yupingyang.sdu@gmail.com}

\date{}
\maketitle

\begin{abstract}
We provide a classification of finite-dimensional graded pointed Majid algebras generated by finite abelian groups as group-like elements and a set of
quasi-commutative skew-primitive elements. This amounts to a classification of finite quasi-quantum linear spaces in the sense of nonassociative geometry.
\end{abstract}

\section{Introduction}
As a further extension of noncommutative geometry, the theory of nonassociative geometry has aroused a lot of interest. The idea is to think of the nonassociative
algebra geometrically as the coordinate algebra of a nonassociative space. Among which, the class of \emph{quasi-associative algebras} and the corresponding
\emph{quasi-quantum geometry} is better understood. The crux of this quasi-associative setting is that coordinate algebras are nonassociative but in a controlled way by means of a multiplicative associator. This influential philosophy initiated from Drinfeld's seminal work \cite{d} of quasi-Hopf algebras, or
quasi-quantum groups, and has been developed into a broader framework \cite{am, m2, bm3} and applied to quite a few subjects such as quantum field theory
\cite{m1, ms1}, noncommutative differential calculus \cite{ms2, bm1, bm2}, and string theory \cite{bhm}. It is also important to note that this novel idea allows us to treat quasi-associative algebras as if they were effectively associative with a help of the theory of tensor categories. As a marvelous example, it is worth to point out that Albuquerque and Majid discovered in \cite{am1} that the famous octonions are in fact associative and commutative in some suitable braided linear Gr-category.

The goal of the present paper is to contribute more concrete quasi-quantum geometric objects. We provide a classification of finite quasi-quantum linear spaces
which are natural extension of the quantum linear spaces studied in \cite{manin, m4, as1}. Throughout, we comply with the terminology of the review paper \cite{m4} of Majid. Recall that, by quantum linear spaces we mean pointed Hopf algebras that arise from the bosonization of braided linear spaces \cite{m1.5} together with group algebras. Another motivation of the paper is the classification problem of finite pointed tensor categories and the corresponding finite quasi-quantum algebras which has been under intensive research in recent years, see for instance
\cite{a, eg, qha1, qha2, qha3}. As a matter of fact, we obtain a class of finite-dimensional pointed coquasi-Hopf algebras, also called Majid algebras,
which may be understood as the coordinate algebras of finite quasi-quantum linear spaces. Consequently, this also provides a class of finite pointed tensor categories presented as the comodule categories of the obtained Majid algebras.

Now we introduce a bit more about the main result of the paper. Recall that a Majid algebra is a coalgebra endowed with a quasi-associative algebra structure
 and a quasi-antipode in a compatible way, for more details see \cite{m0, qha1, qha2, qha3}. A Majid algebra is said to be pointed, if its underlying coalgebra is
  so, that is, the simple comodules are one-dimensional. We call a Majid algebra graded, if it is coradically graded as a coalgebra, and its quasi-algebra
  structure and quasi-antipode maps respect the coradical grading. Throughout, we work over an algebraically closed field $\k$ with zero characteristic. Our
   purpose is to classify all finite-dimensional graded pointed Majid algebras generated by a finite abelian group $G$ as group-like elements and a set of
   skew-primitive elements $\{X_1, \cdots, X_n\}$ satisfying the quasi-commutative property, namely,
\begin{equation}
X_iX_j=q_{j,i}X_jX_i
\end{equation}
for some $q_{j,i}\in \k$ for all $1\leq i\neq j\leq n.$ These are the so-called finite quasi-quantum linear spaces in accordance with the idea of nonassociative
 geometry. The case with $n=2$ was considered in \cite{qqp} by direct computations. Those ideas seem not applicable to this general situation.

The method of classification in the present paper may be viewed as an extension of that used in \cite{as1} for the quantum linear spaces to the quasi-setting.
 Let $G$ be a finite abelian group. Firstly, we collect some necessary facts about the Yetter-Drinfeld category $^{\k G}_{\k G}\mathcal{Y}\mathcal{D}^\Phi$ of the group Majid algebra
 $(\k G, \Phi).$ This category is a generalization of that of Whitehead's $G$-crossed modules \cite{w} and was computed explicitly by Majid in \cite{m3} as an example of his dual or center of the comodule category of $(\k G, \Phi).$ Then we investigate commutative Nichols algebras, or equivalently braided linear spaces, inside $^{\k G}_{\k G}\mathcal{Y}\mathcal{D}^\Phi.$ Finally we determine all finite
 quasi-quantum linear spaces via a quasi-version of Majid's bosonization \cite{m}, or Radford's biproduct \cite{rad}. It is worthy to remark that this method
  may be applied to pursue more general finite-dimensional pointed Majid algebras. The key lies in determining finite Nichols algebras inside those
  Yetter-Drinfeld categories of form $^{\k G}_{\k G}\mathcal{Y}\mathcal{D}^\Phi.$ It is expected that the successful ideas in \cite{as2, as3} may provide
   a useful model to tackle this problem. We will turn to this subject matter in later works.

Here is the layout of the paper. In Section 2, we recall some preliminaries. In Section 3, we provide the general foundation for the study of graded pointed
 Majid algebras via the 3-step method mentioned above. Finally in Sections 4 and 5, we give a classification of finite quasi-quantum linear spaces.

\section{Preliminaries}
In this section, we recall some definitions, notations and basic facts about braided Hopf algebras, the Yetter-Drinfeld categories of group Majid algebras and
 Nichols algebras within them.

\subsection{Braided Hopf algebras}
Let $(\mathcal{C},\otimes, \1, a, l, r, c),$ denoted briefly by $\C$ in the following, be a braided tensor category, where $\1$ is the unit object,
 $a$ ($l,$ or $r$) is the associativity (left, or right unit) constraint and $c$ is the braiding. An (associative) algebra in $\C$ is an object $A$ of
 $\mathcal{C}$ endowed with a multiplication morphism $m: A\otimes A\To A$ and a unit morphism $u: \1 \To A$ such that
 $m\circ(m\otimes \id)=m\circ(\id\otimes m)\circ a_{A,A,A}$  and $m\circ(\id \otimes u)=r_A, \ m\circ(u \otimes \id)=l_A.$ Dually, a (coassociative) coalgebra
 in $\C$ is an object $C$ of $\C$ endowed with a comultiplication  morphism $\Delta: C\To C\otimes C$ and a counit morphism $\varepsilon: C \To \1$ such that
  $a_{C,C,C}\circ(\Delta\otimes \id) \circ \Delta=(\id\otimes \Delta)\circ \Delta$ and
  $r_C^{-1}= (\id \otimes \varepsilon)\circ\Delta, \ l_C^{-1}=(\varepsilon\otimes \id)\circ\Delta.$

If $(A,m_A,u_A)$ and $(B,m_B,u_B)$ are two algebras in $\mathcal{C},$ then one can define a morphism
$m_{A\otimes B}: (A\otimes B)\otimes (A\otimes B)\To A\otimes B$ by
\begin{equation}
m_{A\otimes B}=(m_A\otimes m_B)\circ a_{A\otimes A,B,B}\circ(a_{A,A,B}^{-1}\otimes \id)
\circ(\id\otimes c_{B,A}\otimes \id)\circ(a_{A,B,A}\otimes \id)\circ a_{A\otimes B,A,B}^{-1}.
\end{equation}
Clearly this construction is a natural generalization of the usual tensor product of algebras and was developed in general braided monoidal categories by Majid in his theory of braided groups \cite{m1.5}.

\begin{proposition}
Suppose that $(A,m_A,u_A)$ and $(B,m_B,u_B)$ are algebras in the braided tensor category $\mathcal{C},$ then $(A \otimes B, m_{A\otimes B}, u_A\otimes u_B)$ is an algebra in $\mathcal{C}.$
\end{proposition}

The reader is referred to \cite[Lemma 2.1]{m4} for a proof by means of braid diagrams. The resulting algebra is called the braided tensor product of $A$ and $B.$
Dually, if $(C, \Delta_C, \varepsilon_C)$ and $(D, \Delta_D, \varepsilon_D)$ are coalgebras in $\C,$ then one can define a suitable morphism
$\Delta_{C \otimes D}: C \otimes D \To (C \otimes D) \otimes (C \otimes D)$ such that $(C \otimes D, \Delta_{C \otimes D}, \varepsilon_C \otimes \varepsilon_D)$
 is again a coalgebra in $\C.$

Armed with his construction of braided tensor product, Majid introduced and studied systematically the theory of braided groups, or braided Hopf algebras \cite{m1.5, m4}. For completeness, we record the definition in the following.

\begin{definition}
We say that a sextuplet $(H,m,u,\Delta,\varepsilon,S)$ is a Hopf algebra in the braided tensor category $\mathcal{C},$ or simply a braided Hopf algebra, if $(H,m,u)$ is an algebra in
$\mathcal{C},$  $(H,\Delta,\varepsilon)$ is a coalgebra in $\mathcal{C}$ and $\Delta:H \To H\otimes H$ and $\varepsilon: H \To \1$ are algebra maps
in $\C,$ and $S: H \To H$ is a morphism, to be called the antipode, subject to
\begin{equation}
m \circ (S \otimes \id) \circ \Delta = u \circ \varepsilon = m \circ (\id \otimes S) \circ \Delta.
\end{equation}
\end{definition}

As the usual case, one may naturally define ideals of algebras, coideals of coalgeras, Hopf ideals of Hopf algebras, and the corresponding quotient structures in braided tensor categories. The details are omitted.

\subsection{The Yetter-Drinfeld category of $(\k G,\Phi)$}
Recall that the Yetter-Drinfeld category $^H_H\mathcal{YD}$ of a quasi-Hopf algebra $H$ may be defined as the center $Z(H$-$\mod)$ of its module category
$H$-$\mod$ and it is known to be braided tensor equivalent to the module category of the quantum double $D(H)$ of $H,$ see \cite{majid, m3} for more details.
 We may define the Yetter-Drinfeld category of a finite-dimensional Majid algebra by the canonical duality procedure, namely, the Yetter-Drinfeld category
 of a given finite-dimensional Majid algebra $M$ is defined to be $^{M^*}_{M^*}\mathcal{Y}\mathcal{D}$ where $M^*$ is the quasi-Hopf algebra dual to $M.$

In this paper we mainly concern with the Yetter-Drinfeld category of the group Majid algebra $(\k G,\Phi)$ of a finite abelian group $G$ and a normalized
3-cocycle $\Phi$ on $G.$ To emphasize $\Phi,$ we denote such a category as $_{\k G}^{\k G}\mathcal{Y}\mathcal{D}^{\Phi}.$ Now we recall some more details of
 such categories as given previously in \cite{dpr, majid, m3}. By abuse of notation, we denote the dual of $(\k G,\Phi)$ by $(\k (G), \Phi)$ where the latter
 $\Phi$ is the associator which is obtained by linear extension of the 3-cocycle $\Phi.$  Let $\{\delta_g\}_{g\in G}$ be a basis of $\k(G)$ where $\delta_g$
  is the Dirac function at the point $g.$ By $D^\Phi(G)$ we denote the quantum double of $(\k(G),\Phi).$ Then $D^\Phi(G)$ is a quasitriangular quasi-Hopf algebra
   with product and coproduct determined by
\begin{eqnarray}
\delta_gx\cdot\delta_hy &=& \delta_{g,h}\frac{\Phi(g,x,y)\Phi(x,y,g)}{\Phi(x,g,y)}\delta_gxy, \\
\Delta(\delta_gx)&=&\sum_{hk=g}\frac{\Phi(h,k,x)\Phi(x,h,k)}{\Phi(h,x,k)}\delta_hx\otimes \delta_kx,
\end{eqnarray}
and with associator $\varphi$ and universal $\mathcal{R}$-matrix given by
\begin{eqnarray}
\varphi&=&\sum_{g,h,k\in G}\Phi(g,h,k)^{-1}\delta_g1\otimes \delta_h1\otimes \delta_k1,\\
\mathcal{R}&=&\sum_{g\in G}\delta_g1\otimes \delta_1g.
\end{eqnarray}
Clearly $\sum_{g\in G}\d_g1$ is the identity of $D^\Phi (G)$ by (2.3), so for any $D^\Phi(G)$-module $V$ we have $$V=\oplus_{g\in G}\delta_g1V=\oplus_{g\in G}V_g,$$
 where $\delta_g1V$ is written as $V_g$ for brevity. For any $v \in V_g,$ note that
 \[ \delta_gx\cdot (\delta_gy\cdot v)=\frac{\Phi(g,x,y)\Phi(x,y,g)}{\Phi(x,g,y)}(\delta_gxy)\cdot v \quad \mathrm{and} \quad \delta_hx\cdot v=0 \ \mathrm{if} \ h\neq g. \]
Let $\widetilde{\Phi}_g(x,y)=\frac{\Phi(g,x,y)\Phi(x,y,g)}{\Phi(x,g,y)}$ and by direct computation one can show that $\widetilde{\Phi}_g$ is a 2-cocycle on $G.$ That is to
 say, $V_g$ is a projective $G$-representation with respect to the 2-cocycle $\widetilde{\Phi}_g,$ also called a $(G,\widetilde{\Phi}_g)$-representation,
 see \cite{k, qha2}. We remark that in \cite{qha2} the same symbol $\widetilde{\Phi}_g$ was used to denote a different 2-cocycle, however the final projective
 representation as defined later in Lemma 3.1 is esentially identical to that in \cite{qha2}. Now we summarize some useful properties of $_{\k G}^{\k G}\mathcal{Y}\mathcal{D}^{\Phi}$ in the following proposition, see \cite{m3} for more details.

\begin{proposition}
A vector space $V$ is an object in $_{\k G}^{\k G}\mathcal{Y}\mathcal{D}^{\Phi}$ if and only if $V=\oplus_{g\in G}V_g$ with each $V_g$ a projective
$G$-representation with respect to the 2-cocycle $\widetilde{\Phi}_g,$ namely
\begin{equation}
e\triangleright(f\triangleright v)=\widetilde{\Phi}_g(e,f) (ef)\triangleright v.
\end{equation}
The associativity and the braiding constraints of $_{\k G}^{\k G}\mathcal{Y}\mathcal{D}^{\Phi}$ are given respectively by
\begin{eqnarray}
&\ a_{V_e,V_f,V_g}((X\otimes Y)\otimes Z) =\Phi(e,f,g)^{-1} X\otimes (Y\otimes Z )\\
&R(X\otimes Y)=e \triangleright Y\otimes X
\end{eqnarray}
for all $X\in V_e,\  Y\in V_f,\  Z \in V_g.$
\end{proposition}

\subsection{Nichols algebras in $^{\k G}_{\k G}\mathcal{Y}\mathcal{D}^\Phi$}
Nichols algebras are, roughly speaking, the analogue of the familiar symmetric algebras in more general braided tensor categories. In the classification problem of
 finite-dimensional pointed Hopf algebras, Nichols algebras in $^{\k G}_{\k G}\mathcal{Y}\mathcal{D}^\Phi$ with trivial $\Phi$ play a key role, see for instance
 \cite{as1, as2}. Naturally, in order to tackle the classification problem of finite-dimensional pointed Majid algebras we need to study first Nichols algebras
  in more general Yetter-Drinfeld categories of form $^{\k G}_{\k G}\mathcal{Y}\mathcal{D}^\Phi$ with nontrivial $\Phi.$ An obvious difficulty we have to confront,
   in this situation, is that though associative in the category $^{\k G}_{\k G}\mathcal{Y}\mathcal{D}^\Phi$ these Nichols algebras are generally
    \emph{nonassociative} in the usual sense. To overcome this, we shall take Majid's braided Hopf algebra approach \cite{m1.5, m4} which starts with the tensor algebra of any object in $^{\k G}_{\k G}\mathcal{Y}\mathcal{D}^\Phi$ and by an appropriate quotient to obtain the desired Nichols algebra.

Let $V$ be a nonzero object in $^{\k G}_{\k G}\mathcal{Y}\mathcal{D}^\Phi.$ By $T(V)$ we denote the tensor algebra in $^{\k G}_{\k G}\mathcal{Y}\mathcal{D}^\Phi$
generated by $V,$ that is, the algebra freely generated by $V$ subject to the associative condition
\[ u\otimes (v\otimes w)- a_{V,V,V}\big( (u\otimes v)\otimes w\big), \quad \forall u,v,w\in V.\]
It is clear that $T(V)$ is isomorphic to $\bigoplus_{n \geq 0}V^{\otimes \overrightarrow{n}},$ where $V^{\otimes \overrightarrow{n}}$ means
$\underbrace{(\cdots((}_{n-1}V\otimes V)\otimes V)\cdots \otimes V).$ This induces naturally an $\mathbb{N}$-graded structure on $T(V).$ Define
a comultiplication on $T(V)$ by $\Delta(X)=X\otimes 1+1\otimes X, \ \forall X \in V,$ counit by $\varepsilon(X)=0,$ and antipode by $S(X)=-X.$ It
 is not hard to show that these provide a graded braided Hopf algebra structure on $T(V)$ within the braided tensor category
  $^{\k G}_{\k G}\mathcal{Y}\mathcal{D}^\Phi.$

\begin{definition}
The Nichols algebra $\mathcal{B}(V)$ of $V$ in $^{\k G}_{\k G}\mathcal{Y}\mathcal{D}^\Phi$ is defined to be the quotient braided Hopf algebra $T(V)/I$
where $I$ is the unique maximal graded Hopf ideal generated by homogeneous elements of degree greater than or equal to 2.
\end{definition}

Nichols algebras in the Yetter-Drinfeld category of a Hopf algebra can be defined by various equivalent ways, see \cite{as2}. Here we adopt the method of
definition using the universal property of Nichols algebras. Our definition works for the Yetter-Drinfeld category of any finite-dimensional (co-)quasi-Hopf
 algebras. Needless to say, this definition of Nichols algebras in  $^{\k G}_{\k G}\mathcal{Y}\mathcal{D}^\Phi$ reduces to that in
 $^{\k G}_{\k G}\mathcal{Y}\mathcal{D}$ if $\Phi$ is trivial.

\subsection{Braided linear spaces}
For our purpose, a braided linear space in $^{\k G}_{\k G}\mathcal{Y}\mathcal{D}^\Phi$ will be defined as an $\mathbb{N}$-graded braided Hopf algebra $\mathcal{S}$
within $^{\k G}_{\k G}\mathcal{Y}\mathcal{D}^\Phi$ generated by a set $\{X_i\}_{1\leq i\leq n}$ of primitive elements subject to relations
\begin{gather}
X_i^{\overrightarrow{N_i}}=0\  \mathrm{for \ some \ positive \ integer} \ N_i, \ 1 \leq i \leq n, \\
X_iX_j=q_{j,i}X_jX_i \ \mathrm{for \ all}\ i \neq j,
\end{gather} where $X^{\overrightarrow{N}}$ means $\underbrace{(\cdots((}_{N-1}XX)X)\cdots X).$ Let $V=\mathcal{S}(1)=\oplus_{1\leq i\leq n}\k X_i$ and $V$ is
 an object in $^{\k G}_{\k G}\mathcal{Y}\mathcal{D}^\Phi$ whose actions and coactions are the restriction of those of $\mathcal{S}.$ Later in Theorem 3.6,
 we will show that $\S$ is in fact the Nichols algebra $\mathcal{B}(V)$ of $V$ in $^{\k G}_{\k G}\mathcal{Y}\mathcal{D}^\Phi.$ Note also that $\mathcal{S}$ is
  commutative in the braided tensor category $^{\k G}_{\k G}\mathcal{Y}\mathcal{D}^\Phi.$ By the philosophy of nonassociative geometry, $\mathcal{S}$ may be
  viewed as the ``coordinate algebra" of the linear space $V$ in $_{\k G}^{\k G}\mathcal{Y}\mathcal{D}^{\Phi}.$ In the following we write $\mathcal{S}(V)$ instead
   of $\mathcal{S}$ to emphasize the deep relation between $\mathcal{S}$ and $V.$ We also say $\mathcal{S}(V)$ is of rank $n$ if $\dim V=n.$

\section{Nichols algebras and quasi-quantum groups}
Throughout this section, $M$ is assumed to be a finite-dimensional graded pointed Majid algebra generated by an abelian group $G$ and a set of
skew-primitive elements unless otherwise stated. Hence we may write $M=\oplus_{n\geq 0}M(n)$ such that $M_n:=\oplus_{0\leq i\leq n}M(i)$ is the $n$-th term of
 its coradical filtration. Clearly, the coradical $M_0$ is exactly $\k G,$ and $M_0$ admits a Majid subalgebra structure inherited from $M$ with the associator
  determined by a normalized 3-cocycle $\Phi$ on $G,$ and with an antipode $(S,\alpha,\beta)$ given by
   $S(g)=g^{-1},$ $\alpha(g)=1$ and $\beta(g)=\frac{1}{\Phi(g,g^{-1},g)}$ for all $g\in G.$ On the other hand, the associator $\Phi$ and the antipode
   $(S,\alpha,\beta)$ of $M_0$ can be extended to those for $M,$ see \cite{qha1}. In particular,
 \begin{eqnarray}
 &\Phi(x,y,z)=0,\ \ \ \alpha(x)=\beta(x)=0, \notag \\
 &S(a_1)\alpha(a_2)a_3=\alpha(a),\ \ \ a_1\beta(a_2)S(a_3)=\beta(a),\\
&\Phi\big(a_1,S(a_3),a_5\big)\beta(a_2)\alpha(a_4)=\Phi^{-1}\big(S(a_1),a_3,S(a_5)\big)\alpha(a_2)\beta(a_4)=\varepsilon(a) \notag
\end{eqnarray}
for all $x,y,z\in \oplus_{i\geq 1}M(i)$ and $a\in M.$ Here and below we use Sweedler's sigma notation for the $i$-th iterated comultiplication
$$\D^i(a)=a_1\otimes a_2\otimes \cdots \otimes a_{i+1}.$$

Let $\pi: M \rightarrow M_0$ be the canonical projection. The associated coinvariant subalgebra of $M$ is defined by
\begin{equation}
R:=M^{\operatorname{coinv} M_0}=\{x\in M| (\id\otimes \pi)\Delta(x)=x\otimes 1\}.
\end{equation}
One can easily show that $R$ is closed under the multiplication of $M.$ The main task of this section is to prove that endowed with an appropriate coalgebra
structure $R$ is in fact a braided Hopf algebra, more precisely a Nichols algebra, in the Yetter-Drinfeld category $_{M_0}^{M_0}\mathcal{Y}\mathcal{D}^{\Phi},$ and to establish a quasi-version of Majid's bosonization. We remark that Majid's bosonization was explored in a very broad context in \cite{b} by using braided diagrams without considering nontrivial associators explicitly.  Theoretically, one can always give these form of cross product abstractly without considering associator because of Coherence Theorem and Strictness Theorem. But for our purpose, we need to present our Majid algebras by generators and relations, so these braided diagrams are not very useful for us. Hence the exploration of explicit cross product formulae with nontrivial associators are necessary.

\subsection{$R$ is a braided Hopf algebra in $_{\k G}^{\k G}\mathcal{Y}\mathcal{D}^{\Phi}$}
\begin{lemma}
$R$ is an object in $_{\k G}^{\k G}\mathcal{Y}\mathcal{D}^{\Phi}$.
\end{lemma}

\begin{proof}Notice that $M$ is a $kG$-bicomodule naturally through
$$\delta_{L}:=(\pi\otimes \id)\D,\;\;\;\;\delta_{R}:=(\id\otimes \pi)\D.$$
Thus there is a $G$-bigrading on $M$, that is,
$$M=\bigoplus_{g,h\in G}\;^{g}M^{h}$$
where $^{g}M^{h}=\{m\in M|\delta_{L}(m)=g\otimes m,\;\delta_{R}(m)=m\otimes h\}$. So by definition we have $R=\oplus_{g\in G}\ ^{1}M^{g}=\oplus_{g\in G}R_g,$
 where $R_g=\ ^{1}M^{g}.$ Hence $R$ admits a $G$-graded structure. According to Proposition 2.4, we only need to prove that $R_g$ is a
 $(G,\widetilde{\Phi}_g)$-representation under the $G$-action defined by
\begin{equation} f\triangleright X=\frac{\Phi(fg,f^{-1},f)}{\Phi(f,f^{-1},f)}(f\cdot X)\cdot f^{-1} \end{equation} for
all $f\in G,\ X\in R_g.$ Here $f\cdot X$ means the product in the Majid algebra $M.$

For our purpose, first note that
$$[(f\cdot X)\cdot f^{-1}]\cdot f=\frac{\Phi(f,f^{-1},f)}{\Phi(fg,f^{-1},f)}f\cdot X,$$ hence
\begin{equation}
f\cdot X=\bigg[\frac{\Phi(fg,f^{-1},f)}{\Phi(f,f^{-1},f)}(f\cdot X)\cdot f^{-1}\bigg]\cdot f=(f\triangleright X) \cdot f.
\end{equation}
Then for all $e,f \in G,$ we have
\begin{equation*}
\begin{split}
(ef)\cdot X=&\Phi(e,f,g)^{-1}e\cdot(f\cdot X)\\
=&\Phi(e,f,g)^{-1}e\cdot[(f\triangleright X)\cdot f]\\
=&\Phi(e,f,g)^{-1}\Phi(e,g,f)[e\cdot(f\triangleright X)]\cdot f\\
=&\Phi(e,f,g)^{-1}\Phi(e,g,f)\{[e\triangleright(f\triangleright X)]\cdot e\}\cdot f\\
=&\Phi(e,f,g)^{-1}\Phi(e,g,f)\Phi(g,e,f)^{-1}[e\triangleright(f\triangleright X)]\cdot (ef)\\
=&[(ef)\triangleright X]\cdot (ef),
\end{split}
\end{equation*}
where the last equality follows from (3.4). Finally by comparing the last two terms we observe
\begin{equation*}
e\triangleright(f\triangleright X)=\frac{\Phi(e,f,g)\Phi(g,e,f)}{\Phi(e,g,f)}(ef)\triangleright X=\widetilde{\Phi}_g(e,f)(ef)\triangleright X.
\end{equation*}
This says exactly that $R_g$ is a $(G,\widetilde{\Phi}_g)$-representation.
\end{proof}

In the following, we use the lowercase letters such as $x,$ $x_i$, $x^j$ to present respectively the degrees of the corresponding capital letters $X, X_i, X^j$ which are homogeneous elements in $R.$

\begin{proposition}
$R$ is a braided Hopf algebra in $_{\k G}^{\k G}\mathcal{Y}\mathcal{D}^{\Phi}$ in which
\begin{enumerate}
\item the multiplication $m_R$ is inherited from that of $M,$
\item the comultiplication is defined by $\D_R:\; R\to R\otimes R,\;\;\;\;X\mapsto \Phi(x_1,x_2,x_2^{-1})X_{1}\cdot x_{2}^{-1}\otimes X_{2},$
\item the counit is defined by $\e_R:\; R\to k,\;\;\;\;\e_R:=\e|_R,$ and
\item the antipode is defined by $S_R:\;R\to R,\;\;\;\;X\mapsto \frac{1}{\Phi(x,x^{-1},x)}x\cdot S(X).$
\end{enumerate}
\end{proposition}

\begin{proof}
Firstly we need to show that $m_R, \D_R, \e_R, S_R$ are morphisms in $^{G}_{G}\mathcal{YD}^{\Phi}$. We only do $m_R$ for an example.
 Obviously the product $m_R$ of $R$ is $G$-bigraded. Now we need to prove that $$g\triangleright m_R(X\otimes Y)=m_R(g\triangleright (X\otimes Y))$$ for all $g\in G$
 and $G$-homogeneous elements $X\in R_{x}, Y\in R_{y}.$ On the one hand,
\begin{equation*}
\begin{split}
g\triangleright m_R(X\otimes Y)&=\frac{\Phi(gxy,g^{-1},g)}{\Phi(g,g^{-1},g)}[g(XY)]g^{-1}\\
&=\Phi(xy,g,g^{-1})[g(XY)]g^{-1}.
\end{split}
\end{equation*}
On the other hand,
\begin{equation*}
\begin{split}
&m_R(g\triangleright (X\otimes Y))=\frac{\Phi(g,x,y)\Phi(x,y,g)}{\Phi(x,g,y)}m_R(g\triangleright X\otimes g\triangleright Y)\\
&=\frac{\Phi(g,x,y)\Phi(x,y,g)\Phi(x,g,g^{-1})\Phi(y,g,g^{-1})}{\Phi(x,g,y)}[(gX)g^{-1}][(gY)g^{-1}]\\
&=\frac{\Phi(g,x,y)\Phi(x,y,g)\Phi(x,g,g^{-1})\Phi(y,g,g^{-1})\Phi(g^{-1},gy,g^{-1})\Phi(g^{-1},g,y)\Phi(gx,y,g^{-1})}
{\Phi(x,g,y)\Phi(gx,g^{-1},y)\Phi(g,x,y)\Phi(g^{-1},g,g^{-1})}[g(XY)]g^{-1}\\
&=\frac{\Phi(x,y,g)\Phi(x,g,g^{-1})\Phi(y,g,g^{-1})\Phi(g^{-1},gy,g^{-1})\Phi(g^{-1},g,y)\Phi(gx,y,g^{-1})}
{\Phi(x,g,y)\Phi(gx,g^{-1},y)\Phi(g^{-1},g,g^{-1})}[g(XY)]g^{-1}
\end{split}
\end{equation*}
So it suffices to verify
\begin{equation*}
\begin{split}
&\frac{\Phi(x,y,g)\Phi(x,g,g^{-1})\Phi(y,g,g^{-1})\Phi(g^{-1},gy,g^{-1})\Phi(g^{-1},g,y)\Phi(gx,y,g^{-1})}
{\Phi(x,g,y)\Phi(gx,g^{-1},y)\Phi(g^{-1},g,g^{-1})}=\Phi(xy,g,g^{-1}),
\end{split}
\end{equation*}
which follows from
\begin{equation*}
\begin{split}
&\frac{\Phi(x,y,g)\Phi(x,g,g^{-1})\Phi(y,g,g^{-1})\Phi(g^{-1},gy,g^{-1})\Phi(g^{-1},g,y)\Phi(gx,y,g^{-1})}
{\Phi(x,g,y)\Phi(gx,g^{-1},y)\Phi(g^{-1},g,g^{-1})\Phi(xy,g,g^{-1})}\\
&=\frac{\Phi(x,g,g^{-1})\Phi(g^{-1},gy,g^{-1})\Phi(g^{-1},g,y)\Phi(gx,y,g^{-1})}
{\Phi(x,g,y)\Phi(gx,g^{-1},y)\Phi(g^{-1},g,g^{-1})\Phi(x,gy,g^{-1})}\\
&=\frac{\Phi(x,g,g^{-1})\Phi(g^{-1},gy,g^{-1})\Phi(g^{-1},g,y)\Phi(gx,y,g^{-1})\Phi(g,y,g^{-1})}
{\Phi(x,g,y)\Phi(gx,g^{-1},y)\Phi(g^{-1},g,g^{-1})\Phi(x,gy,g^{-1})\Phi(g,y,g^{-1})}\\
&=\frac{\Phi(x,g,g^{-1})\Phi(gx,y,g^{-1})\Phi(g^{-1},g,yg^{-1})}
{\Phi(gx,g^{-1},y)\Phi(g^{-1},g,g^{-1})\Phi(xg,y,g^{-1})\Phi(x,g,yg^{-1})}\\
&=\frac{\Phi(x,g,g^{-1})\Phi(g^{-1},g,yg^{-1})}
{\Phi(g^{-1},g,g^{-1})\Phi(x,g,g^{-1})\Phi(g,g^{-1},y)}\\
&=\frac{\Phi(g^{-1},g,yg^{-1})}
{\Phi(g^{-1},g,g^{-1})\Phi(g,g^{-1},y)}\\
&=1.
\end{split}
\end{equation*}
The verification for $\D_R, \e_R, S_R$ are similar and so omitted.

Next we will show that $R$ is an algebra and a coalgebra in $_{\k G}^{\k G}\mathcal{Y}\mathcal{D}^{\Phi}.$ For any $G$-homogeneous $X,Y,Z\in R,$ we have
 $(X\cdot Y)\cdot Z=\Phi(x,y,z)^{-1}X\cdot (Y\cdot Z).$ This clearly endows $R$ an algebra structure in $_{\k G}^{\k G}\mathcal{Y}\mathcal{D}^{\Phi}.$
So here we need to verify that $\D_R$ is coassociative in the category, that is, $a\circ(\D_R\otimes \id) \circ \D_R=(\id\otimes \D_R) \circ \D_R.$

In fact, we have
\begin{eqnarray*}
(\id\otimes \D_R)\D_R(X)&=&\Phi(x_1,x_2,x_2^{-1})X_1x_{2}^{-1}\otimes \D_R(X_2)\\
&=&\Phi(x_1,x_2x_3,(x_2x_3)^{-1})\Phi(x_2,x_3,x_3^{-1})X_1(x_{2}x_3)^{-1}\otimes (X_2x_3^{-1}\otimes X_3)
\end{eqnarray*}
and
\begin{eqnarray*}
(\D_R\otimes \id)\D_R(X)&=&\Phi(x_1,x_2,x_2^{-1})\D_R(X_1x_{2}^{-1})\otimes X_2\\
&=&\Phi(x_{11}x_{12},x_2,x_2^{-1})\Phi(x_{11},x_{12},x_{12}^{-1})[(X_{11}x_{2}^{-1})x_{12}^{-1}\otimes X_{12}x_2^{-1}]\otimes X_2\\
&=&\frac{\Phi(x_1x_2,x_3,x_3^{-1})\Phi(x_1,x_2,x_2^{-1})\Phi(x_2x_3,x_3^{-1},x_2)}
{\Phi(x_1x_2x_3,x^{-1}_3,x_2^{-1})} (X_1(x_{2}x_3)^{-1}\otimes X_2x_3^{-1})\otimes X_3.
\end{eqnarray*}
So to show that $\Phi\circ(\D_R\otimes \id)\D_R(X)=(\id\otimes \D_R)\D_R(X)$, it is enough to prove that
\begin{equation}\Phi(x_1,x_2x_3,(x_2x_3)^{-1})\Phi(x_2,x_3,x_3^{-1})=\frac{\Phi(x_1x_2,x_3,x_3^{-1})\Phi(x_1,x_2,x_2^{-1})\Phi(x_2x_3,x_3^{-1},x_2)}
{\Phi(x_1x_2x_3,x^{-1}_3,x_2^{-1})\Phi(x_1,x_2,x_3)}.
\end{equation}
 The fact $\partial(\Phi)(x_1,x_2x_3,x_3^{-1},x_2^{-1})=1$ implies that
$$\frac{\Phi(x_2x_3,x_3^{-1},x_2^{-1})\Phi(x_1,x_2,x_2^{-1})\Phi(x_1,x_2x_3,x_3^{-1})}{\Phi(x_1x_2x_3,x_3^{-1},x_2^{-1})
\Phi(x_1,x_2x_3,(x_2x_3)^{-1})}=1.$$
Then to show (3.5), it is enough to prove that
$$\Phi(x_2,x_3,x_3^{-1})=\frac{\Phi(x_1x_2,x_3,x_3^{-1})}{\Phi(x_1,x_2x_3,x_3^{-1})\Phi(x_1,x_2,x_3)}.$$
But this is a direct consequence of $\partial(\Phi)(x_1,x_2,x_3,x_3^{-1})=1$.

Thirdly we need to show that $\D_R$ is an algebra morphism in $^{G}_{G}\mathcal{YD}^{\Phi}$.

On the one hand, we have
\begin{eqnarray*}
\D_R(XY)&=&\Phi(x_1y_1,x_2y_2,(x_2y_2)^{-1})(X_1Y_1)(x_{2}y_2)^{-1}\otimes X_2Y_2.
\end{eqnarray*}
On the other hand,

$\D_R(X)\D_R(Y)$\begin{eqnarray*}&=&\Phi(x_1,x_2,x_2^{-1})\Phi(y_1,y_2,y_2^{-1})(X_1x_{2}^{-1}\otimes X_2)(Y_1y_2^{-1}\otimes Y_2)\\
&=&\frac{\Phi(x_1,x_2,x_2^{-1})\Phi(y_1,y_2,y^{-1}_2)\Phi(x_1x_2,y_1,y_2)\Phi(x_1,y_1,x_2)\Phi(y_1,x_2,x^{-1}_2)\Phi(x_1,x_2,y_1x_2^{-1})}
{\Phi(x_1y_1,x_2,y_2)\Phi(x_2,y_1,x_2^{-1})\Phi(x_1,x_2,y_1)\Phi(x_1x_2,x^{-1}_2,x_2)}\\
&&\times \frac{\Phi(x_2,x^{-1}_2,x_2)\Phi(y_2,y^{-1}_2,x^{-1}_2)\Phi(x_1x_2,y_1y_2,(x_2y_2)^{-1})}
{\Phi(y_1y_2,y^{-1}_2,x^{-1}_2)\Phi(x_2,y_2,(x_2y_2)^{-1})}(X_1Y_1)(x_{2}y_2)^{-1}\otimes X_2Y_2.
\end{eqnarray*}
The the desired equality $\D_R(XY)=\D_R(X)\D_R(Y)$ follows from
\begin{eqnarray*}&&\frac{\Phi(x_1,x_2,x_2^{-1})\Phi(y_1,y_2,y^{-1}_2)\Phi(x_1x_2,y_1,y_2)\Phi(x_1,y_1,x_2)\Phi(y_1,x_2,x^{-1}_2)\Phi(x_1,x_2,y_1x_2^{-1})}
{\Phi(x_1y_1,x_2,y_2)\Phi(x_2,y_1,x_2^{-1})\Phi(x_1,x_2,y_1)\Phi(x_1x_2,x^{-1}_2,x_2)}\\
&&\times \frac{\Phi(x_2,x^{-1}_2,x_2)\Phi(y_2,y^{-1}_2,x^{-1}_2)\Phi(x_1x_2,y_1y_2,(x_2y_2)^{-1})}
{\Phi(y_1y_2,y^{-1}_2,x^{-1}_2)\Phi(x_2,y_2,(x_2y_2)^{-1})}\\
&=& \frac{\Phi(y_1,y_2,y_2^{-1})\Phi(x_1x_2,y_1,y_2)\Phi(x_1,y_1,x_2)\Phi(y_1,x_2,x^{-1}_2)\Phi(x_1,x_2,y_1x_2^{-1})}
{\Phi(x_1y_1,x_2,y_2)\Phi(x_2,y_1,x_2^{-1})\Phi(x_1,x_2,y_1)}\\
&&\times \frac{\Phi(y_2,y^{-1}_2,x^{-1}_2)\Phi(x_1x_2,y_1y_2,(x_2y_2)^{-1})}
{\Phi(y_1y_2,y^{-1}_2,x^{-1}_2)\Phi(x_2,y_2,(x_2y_2)^{-1})}\\
&=& \frac{\Phi(x_1x_2,y_1,y_2)\Phi(x_1,y_1,x_2)\Phi(y_1,x_2,x^{-1}_2)\Phi(x_1,x_2,y_1x_2^{-1})}
{\Phi(x_1y_1,x_2,y_2)\Phi(x_2,y_1,x_2^{-1})\Phi(x_1,x_2,y_1)}\\
&&\times \frac{\Phi(y_1,y_2,(y_2x_2)^{-1})\Phi(x_1x_2,y_1y_2,(x_2y_2)^{-1})}
{\Phi(x_2,y_2,(x_2y_2)^{-1})}\\
&=& \frac{\Phi(x_1,y_1,x_2)\Phi(y_1,x_2,x^{-1}_2)\Phi(x_1,x_2,y_1x_2^{-1})\Phi(x_1x_2y_1,y_2,(y_2x_2)^{-1})\Phi(x_1x_2,y_1,x_2^{-1})}
{\Phi(x_1y_1,x_2,y_2)\Phi(x_2,y_1,x_2^{-1})\Phi(x_1,x_2,y_1)\Phi(x_2,y_2,(x_2y_2)^{-1})}\\
&=&\frac{\Phi(x_1y_1,x_2,x_2^{-1})\Phi(x_1x_2y_1,y_2,(y_2x_2)^{-1})}
{\Phi(x_1y_1,x_2,y_2)\Phi(x_2,y_2,(x_2y_2)^{-1})}\\
&=&\Phi(x_1y_1,x_2y_2,(x_2y_2)^{-1}).
\end{eqnarray*}

Finally we prove that $S_R$ is an antipode of $R$ in the category, that is, we need to verify identity (2.2).

Indeed, we have
\begin{eqnarray*}
(\id\ast S_R)(X)&=&\Phi(x_1,x_2,x_2^{-1})(X_1x_{2}^{-1}) S_R(X_2)\\
&=&\frac{\Phi(x_1,x_2,x_2^{-1})}{\Phi(x_2,x^{-1}_2,x_2)}(X_1x_{2}^{-1})(x_{2}S(X_2))\\
&=&\frac{\Phi(x_1,x_2,x_2^{-1})\Phi(x_2,x_2^{-1},x_2)}{\Phi(x_2,x^{-1}_2,x_2)\Phi(x_1x_2,x_2^{-1},x_2)}X_1S(X_2)\\
&=&\frac{\Phi(x_1,x_2,x_2^{-1})\Phi(x_2,x_2^{-1},x_2)}{\Phi(x_1x_2,x^{-1}_2,x_2)}X_1\frac{1}{\Phi(x_2,x_2^{-1},x_2)}S(X_2)\\
&=&X_1\frac{1}{\Phi(x_2,x_2^{-1},x_2)}S(X_2)\\
&=&\beta(X)=0=\e(X)
\end{eqnarray*}
for all $X\in R_{\geq 1}$. Similarly, one can show that $S_R\ast \id=\e_R$.

We complete the proof of the proposition.
\end{proof}

\subsection{A quasi-version of Majid's bosonization}
The aim of this subsection is to give a quasi-version of the well-known Majid's bosonization. Let $H$ be a braided Hopf algebra in $_{\k G}^{\k G}\mathcal{Y}\mathcal{D}^{\Phi}.$
 For our purpose, we assume further that $H$ is $\mathbb{N}$-graded with $H(0)=\k$ and $H$ is generated by $H(1)$ as an algebra. We will say that $X \in H(n)$ has
  length $n.$ Note that $R$ and $B(V)$ are natural examples of such braided Hopf algebras in $_{\k G}^{\k G}\mathcal{Y}\mathcal{D}^{\Phi}.$ In the following, a
  homogeneous element means it is both $\mathbb{N}$- and $G$-homogeneous.

\begin{proposition}
Keep the assumptions on $H$ as above. Define on $H\otimes\k G$ a product by
\begin{equation}
(X\otimes g)(Y\otimes h)=\frac{\Phi(xg,y,h)\Phi(x,y,g)}{\Phi(x,g,y)\Phi(xy,g,h)}X(g\triangleright Y)\otimes gh,
\end{equation}
and a coproduct  by
\begin{equation}
\b(X\otimes g)=\Phi(x_1,x_2,g)^{-1}(X_1\otimes x_2g)\otimes (X_2\otimes g),
\end{equation}
then $H\otimes\k G$ becomes a graded Majid algebra with an antipode $(S,\alpha,\beta)$ given by
\begin{eqnarray}
&S(X\otimes g)=\frac{\Phi(g^{-1},g,g^{-1})}{\Phi(x^{-1}g^{-1},xg,g^{-1})\Phi(x,g,g^{-1})}(1\otimes x^{-1}g^{-1})(S_H(X)\otimes 1), \\
&\alpha(1\otimes g)=1,\ \ \ \alpha(X\otimes g)=0,  \\
&\beta(1\otimes g)=\Phi(g,g^{-1},g)^{-1},\ \ \beta(X\otimes g)=0,
\end{eqnarray}
where $g,h\in G$ and $X,Y$ are homogeneous elements of length  $\geq 1.$
By $H\#\k G$ we denote the Majid algebra so defined on $H\otimes \k G.$
\end{proposition}

\begin{proof}
First we show that $H \otimes \k G$ is a coalgebra under the given comultiplication. By direct computation we have
\begin{equation*}
\begin{split}
(\b\otimes \id)\b(X\otimes g)=&(\b\otimes \id)\big(\Phi(x_1,x_2,g)^{-1}(X_1\otimes x_2g)\otimes (X_2\otimes g)\big)\\
=&\frac{1}{\Phi(x_1,x_2,x_3g)\Phi(x_1x_2,x_3,g)}(X_1\otimes x_2x_3g)\otimes (X_2\otimes x_3g)\otimes (X_3\otimes g)
\end{split}
\end{equation*}
and
\begin{equation*}
\begin{split}
(\id\otimes \b)\b(X\otimes g)=&(\b\otimes \id)\big(\Phi(x_1,x_2,g)^{-1}(X_1\otimes x_2g)\otimes (X_2\otimes g)\big)\\
=&\frac{1}{\Phi(x_1,x_2x_3,g)\Phi(x_2,x_3,g)}(X_1\otimes x_2 x_3g)\otimes (X_2\otimes x_3g)\otimes (X_3\otimes g).
\end{split}
\end{equation*}
Then $(\b\otimes \id)\b(X\otimes g)=(\id\otimes \b)\b(X\otimes g)$ follows from $X_1\otimes (X_2\otimes X_3)=\Phi(x_1,x_2,x_3)(X_1\otimes X_2)\otimes X_3$ and
 the definition of 3-cocycles.

Next we show that $H \otimes \k G$ is a quasi-algebra under the given multiplication, i.e.,
\begin{equation}
[(X\otimes g)(Y\otimes h)](Z\otimes e)=\frac{\Phi(g,h,e)}{\Phi(xg,yh,ze)}(X\otimes g)[(Y\otimes h)(Z\otimes e)],
\end{equation}
for all homogeneous $X,Y,Z\in H$ and $e,f,g\in G.$
By a direct computation we have
\begin{eqnarray*}
[(X\otimes g)(Y\otimes h)](Z\otimes e)&=&\frac{\Phi(xg,y,h)\Phi(x,y,g)}{\Phi(x,g,y)\Phi(xy,g,h)}(X(g\triangleright Y)\o gh)(Z\o e)\\
&=&\frac{\Phi(xg,y,h)\Phi(x,y,g)}{\Phi(x,g,y)\Phi(xy,g,h)}\frac{\Phi(xygh,z,e)\Phi(xy,z,gh)}{\Phi(xy,gh,z)\Phi(xyz,gh,e)}[X(g\triangleright Y)]
(gh\triangleright Z)\o ghe\\
&=&\frac{\Phi(xg,y,h)\Phi(x,y,g)}{\Phi(x,g,y)\Phi(xy,g,h)}\frac{\Phi(xygh,z,e)\Phi(xy,z,gh)}{\Phi(xy,gh,z)\Phi(xyz,gh,e)}\frac{\Phi(g,z,h)}{\Phi(z,g,h)\Phi(g,h,z)}\\
&&\times [X(g\triangleright Y)](g\triangleright(h\triangleright Z))\o ghe
\end{eqnarray*}
and
\begin{eqnarray*}
(X\otimes g)[(Y\otimes h)(Z\otimes e)]&=&\frac{\Phi(yh,z,e)\Phi(y,z,h)}{\Phi(y,h,z)\Phi(yz,h,e)}(X\o g)(Y(h\triangleright Z)\o he)\\
&=&\frac{\Phi(yh,z,e)\Phi(y,z,h)}{\Phi(y,h,z)\Phi(yz,h,e)}\frac{\Phi(xg,yz,he)\Phi(x,yz,g)}{\Phi(x,g,yz)\Phi(xyz,g,he)}X[g\triangleright(Y(h\triangleright Z))]\o
ghe\\
&=&\frac{\Phi(yh,z,e)\Phi(y,z,h)}{\Phi(y,h,z)\Phi(yz,h,e)}\frac{\Phi(xg,yz,he)\Phi(x,yz,g)}{\Phi(x,g,yz)\Phi(xyz,g,he)}\frac{\Phi(y,z,g)\Phi(g,y,z)}
{\Phi(y,g,z)}\Phi(x,y,z)\\
&& \times [X(g\triangleright Y)](g\triangleright(h\triangleright Z))\o ghe.
\end{eqnarray*}
So (3.11) is equivalent to the following
\begin{eqnarray}
&&\frac{\Phi(yh,z,e)\Phi(y,z,h)}{\Phi(y,h,z)\Phi(yz,h,e)}\frac{\Phi(xg,yz,he)\Phi(x,yz,g)}{\Phi(x,g,yz)\Phi(xyz,g,he)}\frac{\Phi(g,z,h)}{\Phi(z,g,h)\Phi(g,h,z)}\\
&=&\frac{\Phi(g,h,e)}{\Phi(xg,yh,ze)}\frac{\Phi(yh,z,e)\Phi(y,z,h)}{\Phi(y,h,z)\Phi(yz,h,e)}\frac{\Phi(xg,yz,he)\Phi(x,yz,g)}{\Phi(x,g,yz)\Phi(xyz,g,he)}
\frac{\Phi(y,z,g)\Phi(g,y,z)}{\Phi(y,g,z)}\Phi(x,y,z). \notag
\end{eqnarray}

In order to verify this complicated equality, we need to apply the following commutative diagram which holds in any abstract braided tensor category $(\C,a,c)$, here $a$ is the associative constraint and $c$ is the braiding.
\begin{equation}
\xymatrix@C=2.5cm{ [(A\otimes B)\otimes (A\otimes B)]\otimes (A\otimes B)\ar[r]^{a_{A\otimes B,A\otimes B,A\otimes B}}
\ar[d]_{f_{A,B,A,B}\otimes \id_{A\otimes B}}& (A\otimes B)\otimes [(A\otimes B)\otimes (A\otimes B)]\ar[d]^{\id_{A\otimes B}\otimes f_{A,B,A,B}}& \\
   [(A\otimes A)\otimes (B\otimes B)]\otimes (A\otimes B)\ar[d]_{f_{A\otimes A,B\otimes B,A,B}} &
   (A\otimes B)\otimes [(A\otimes A)\otimes (B\otimes B)]\ar[d]^{f_{A,B ,A\otimes A,B\otimes B}}&\\
  [(A\otimes A)\otimes A]\otimes [(B\otimes B)\otimes B]\ar[r]^{a_{A,A,A}\otimes a_{B,B,B}}&[A\otimes (A\otimes A)]\otimes [B\otimes (B\otimes B)]£¬}
\end{equation} for all objects $A,B \in \C.$
Here $f_{C,D,E,F}:(C\otimes D)\otimes(E\otimes F) \To (C\otimes E)\otimes(D\otimes F)$ is an isomorphism defined by
\begin{equation}
f_{C,D,E,F}=a_{C\otimes D,E,F}\circ(a_{C,D,E}^{-1}\otimes \id) \circ (\id\otimes c_{D,E}\otimes \id)\circ(a_{C,E,D}\otimes \id)\circ a_{C\otimes E,D,F}^{-1}
\end{equation}
for all $C,D,E,F$ in $\mathcal{C}.$
The commutativity of the diagram follows from the Coherence Theorem of braided tensor categories, see \cite[Corollary 2.6]{js}.

Let $V \in {_{\k G}^{\k G}\mathcal{Y}\mathcal{D}^{\Phi}}$ with $V_g,V_h,V_e\neq 0.$ Replace $A$ by $H$ and $B$ by $V$ in the diagram (3.13). Choose nonzero elements
 $X'\in V_g,Y'\in V_h, Z'\in V_e.$ Then by a direct computation we have
\begin{eqnarray*}
&&( a_{H,H,H}\o a_{V,V,V})\circ f_{H\o H,V\o V,H,V}\circ (f_{H,V,H,V}\o \id_{H\o V})([(X\o X')\o (Y\o Y')]\o (Z\o Z'))\\
&=&\frac{\Phi(yh,z,e)\Phi(y,z,h)}{\Phi(y,h,z)\Phi(yz,h,e)}( a_{H,H,H}\o a_{V,V,V})\circ f_{H\o H,V\o V,H,V}([(X\o g\triangleright Y)\o (X'\o Y')]\o (Z\o Z'))\\
&=&\frac{\Phi(yh,z,e)\Phi(y,z,h)}{\Phi(y,h,z)\Phi(yz,h,e)}\frac{\Phi(xygh,z,e)\Phi(xy,z,gh)}{\Phi(xy,gh,z)\Phi(xyz,gh,e)}\\
&&\times ( a_{H,H,H}\o a_{V,V,V})([(X\o g\triangleright Y)\o gh\triangleright Z] \o  [(X'\o Y')\o Z')])\\
&=&\frac{\Phi(yh,z,e)\Phi(y,z,h)}{\Phi(y,h,z)\Phi(yz,h,e)}\frac{\Phi(xygh,z,e)\Phi(xy,z,gh)}{\Phi(xy,gh,z)\Phi(xyz,gh,e)}\Phi^{-1}(x,y,z)\Phi^{-1}(g,h,e)\\
&&\times ([X\o (g\triangleright Y\o gh\triangleright Z)] \o  [X'\o (Y'\o Z')]).\\
&=&\frac{\Phi(yh,z,e)\Phi(y,z,h)}{\Phi(y,h,z)\Phi(yz,h,e)}\frac{\Phi(xygh,z,e)\Phi(xy,z,gh)}{\Phi(xy,gh,z)\Phi(xyz,gh,e)}\frac{\Phi(g,z,h)}{\Phi(z,g,h)
\Phi(g,h,z)}\Phi^{-1}(x,y,z)\Phi^{-1}(g,h,e)\\
&&\times ([X\o (g\triangleright Y\o g\triangleright(h\triangleright Z))] \o  [X'\o (Y'\o Z')]).
\end{eqnarray*}
Similarly we have
\begin{eqnarray*}
&&f_{H,V,H\o H,V\o V}\circ (\id_{H\o V}\o f_{H,V,H,V})\circ a_{H\o V,H\o V,H\o V}([(X\o X')\o (Y\o Y')]\o (Z\o Z'))\\
&=&\Phi^{-1}(xg,yh,ze)\frac{\Phi(yh,z,e)\Phi(y,z,h)}{\Phi(y,h,z)\Phi(yz,h,e)}\frac{\Phi(xg,yz,he)\Phi(x,yz,g)}{\Phi(x,g,yz)\Phi(xyz,g,he)}
\frac{\Phi(y,z,g)\Phi(g,y,z)}{\Phi(y,g,z)}\\
&&\times ([X\o (g\triangleright Y\o g\triangleright(h\triangleright Z))] \o  [X'\o (Y'\o Z')]).
\end{eqnarray*}

So the previous two identities are equal. By comparing their coefficients, (3.12) follows, and so does (3.11).

The proof of the fact that the product defined by (3.9) is a coalgebra map follows by a direct verification and so we omit the details.

Finally we prove that $(S,\alpha,\beta)$ defined by (3.8-3.10) is a quasi-antipode of $H\otimes \k G,$  i.e., the identities of (3.1) hold. It is enough to verify
them for elements $X\otimes g$ with $\operatorname{length}(X) \ge 1.$ Here we only prove $ S(a_1)\alpha(a_2)a_3=\alpha(a)\ \mathrm{and}\ a_1\beta(a_2)S(a_3)=\beta(a)$
since other identities are obvious.

Write $(\b\otimes \id) \circ \b (X\otimes g)=\sum(X\otimes g)_1\otimes(X\otimes g)_2\otimes (X\otimes g)_3.$ Note that $\alpha((X\otimes g)_2) \neq 0$ (
 $\beta((X\otimes g)_2)\neq 0$) if and only if $(X\otimes g)_2$ is a scalar of a group-like element, hence
\begin{equation*}
\begin{split}
&S((X\otimes g)_1)\alpha((X\otimes g)_2)(X\otimes g)_3\\
=&\Phi(x_1,x_2,g)^{-1}S(X_1\otimes x_2g)(X_2\otimes g)\\
=&\frac{\Phi(x_2^{-1}g^{-1},x_2g,x_2^{-1}g^{-1})}
{\Phi(x_1,x_2,g)\Phi(x^{-1}g^{-1},xg,x_2^{-1}g^{-1})\Phi(x_1,x_2^{-1}g^{-1},x_2^{-1}g^{-1})}\\
&\times [(1\otimes x^{-1}g^{-1})(S_H(X_1)\otimes 1)](X_2\otimes g).\\
=&\frac{\Phi(x_2^{-1}g^{-1},x_2g,x_2^{-1}g^{-1})\Phi(x_1,x_2,g)}
{\Phi(x_1,x_2,g)\Phi(x^{-1}g^{-1},xg,x_2^{-1}g^{-1})\Phi(x_1,x_2^{-1}g^{-1},x_2^{-1}g^{-1})
\Phi(x^{-1}g^{-1},x_1,x_2g)}\\
&\times (1\otimes x^{-1}g^{-1})(S_H(X_1)X_2\otimes g)\\
=&(1\otimes x^{-1}g^{-1})(S_H(X_1)X_2\otimes g)\\
=&(1\otimes x^{-1}g^{-1})(\varepsilon(X)\otimes g)\\
=&0.
\end{split}
\end{equation*}
So we have proved that $S((X\otimes g)_1)\alpha((X\otimes g)_2)(X\otimes g)_3=\alpha (X\otimes g).$ In a similar manner one can prove
$(X\otimes g)_1\beta((X\otimes g)_2)S((X\otimes g)_3)=\beta(X\otimes g).$

We complete the proof of the proposition.
\end{proof}

\begin{proposition}
Let $M$ and $R$ be as before, and $R\#kG$ the Majid algebra as defined in the previous proposition. Then $R\#kG \cong M.$
\end{proposition}
\begin{proof}

 Define a map
\begin{eqnarray*}
F:R\#\k G &\To &  M\\
X\otimes g &\to& X\cdot g.
\end{eqnarray*}
It is clear that $F$ is a quasi-algebra map. Moreover, as a linear map $F$ is injective by the assumptions on $M$ and $R.$
For any $X\in {^{g}M^{h}}$, clearly $Xh^{-1} \in R$. This implies that $F$ is surjective and hence $F$ is bijective. We verify in the following that $F$ is
also a coalgebra map, i.e.
\begin{equation}
\b \circ F(X\otimes g)=(F\otimes F) \circ \b(X\otimes g).
\end{equation} Then $F$ is in fact an isomorphism of Majid algebras.

By direct calculation, the left hand side of (3.15) is $$\b \circ F(X\otimes g)=\b(X\cdot g)=\b(X)\b(g)=(X_1\cdot x_2)\cdot g\otimes X_2\cdot g=\Phi(x_1,x_2,g)^{-1}
X_1\cdot (x_2g)\otimes X_2\cdot g,$$ and the right hand side is
\begin{equation*}
\begin{split}
(F\otimes F) \circ \b(X\otimes g)=&(F\otimes F)(\Phi(x_1,x_2,g)^{-1}(X_1\otimes x_2g)\otimes (X_2\otimes g)) \\
=&\Phi(x_1,x_2,g)^{-1}X_1(x_2g)\otimes X_2g.
\end{split}
\end{equation*}
Hence (3.15) is proved, i.e., $F$ is a coalgebra map.

It is obvious that $F$ preserves the quasi-antipodes of $R\# \k G$ and $M.$  So $F \colon R\# \k G \to M$ is the desired isomorphism of Majid algebras.
\end{proof}

\begin{remark}
 This proposition can be seen as a quasi-version of Majid's bosonization. Needless to say, if the 3-cocycle $\Phi$ is trivial, then we recover the corresponding result of
 Majid's bosonization.
\end{remark}

\subsection{$R$ is a Nichols algebra in $_{\k G}^{\k G}\mathcal{Y}\mathcal{D}^{\Phi}$}
\begin{theorem}
Keep the assumptions on $M$ and $R.$ Then $R$ is a Nichols algebra in $_{\k G}^{\k G}\mathcal{Y}\mathcal{D}^{\Phi}.$
\end{theorem}

\begin{proof}
Choose a basis $\{X_i\}$ of $R(1)$ such that the $G$-degree of $X_i$ is $g_i.$ Denote $V=R(1)$ and note that $R$ is generated by $V.$ On the other hand, note that
 $V$ is a natural object in $_{\k G}^{\k G}\mathcal{Y}\mathcal{D}^{\Phi}$ and there is a Nichols algebra $B(V)$ in the category. By the universal property of
  $B(V),$ it is clear that there is an epimorphism of  braided Hopf algebras $F:R \to B(V)$ determined by $F(X_i)=X_i.$ This induces an epimorphism of Majid
  algebras \begin{eqnarray*}
\widetilde{F}: R\#\k G &\To& B(V)\#\k G,\\
                X_i \otimes g&\to & X_i \otimes g.
\end{eqnarray*}
Since $R\#\k G$ and $B(V)\#\k G$ are ordinary coalgebras, and $\widetilde{F}$ is injective in group-like elements and skew-primitive elements, hence so is the
restriction of $\widetilde{F}$ to the first term of the coradical filtration of $R\#\k G.$ Now by \cite[Theorem 5.3.1]{mon}, $\widetilde{F}$ is injective.
Therefore, $\widetilde{F}$ is a bijection and hence $F: R \to B(V)$ is an isomorphism by comparing dimensions.
\end{proof}

\section{Braided linear spaces in $_{\k G}^{\k G}\mathcal{Y}\mathcal{D}^{\Phi}$}
In this section, we give a classification of braided linear spaces in $_{\k G}^{\k G}\mathcal{Y}\mathcal{D}^{\Phi}.$
We need some notations about Gaussian binomials in our exposition. For any $\hbar \in \k$, define
$l_\hbar=1+\hbar+\cdots +\hbar^{l-1}$ and $l!_\hbar=1_\hbar \cdots l_\hbar$. The Gaussian binomial coefficient is defined by
 $\binom{l+m}{l}_\hbar:=\frac{(l+m)!_\hbar}{l!_\hbar m!_\hbar}$.

\subsection{Normalized 2-cocycles and 3-cocycles on $G$} Recall that a normalized 2-cocycle on a group $G$ is a function $\omega: G\times G\To \k^*$ such that
$\omega(e,f)\omega(ef,g)=\omega(e,fg)\omega(f,g)$ and $\omega(1,f)=\omega(e,1)=1$ for all $e,f,g\in G.$ A 2-cocycle $\omega$ of $G$ is called symmetric
if $\omega(e,f)=\omega(f,e)$ for all $e,f\in G.$

A function $\Phi:G\times G\times G\mapsto \k^*$ is called a 3-cocycle on $G$ if
\begin{equation}
\Phi(ef,g,h)\Phi(e,f,gh)=\Phi(e,f,g)\Phi(e,fg,h)\Phi(f,g,h)
\end{equation}
for all $e,f,g,h \in G$, and it is called normalized if $\Phi(f,1,g)=1.$  Let $G=\Z_{m_1}\times\Z_{m_2}\times \cdots \times \Z_{m_N}$
and $e_i$ a fixed generator of $\Z_{m_i}$. Denote by $A$ the set of all integer sequences
\begin{equation}
(a_1,\cdots,a_l,\cdots,a_N,a_{12}\cdots,a_{ij},\cdots,a_{N-1,N},a_{123}\cdots,a_{rst},\cdots,a_{N-2,N-1,N})
\end{equation}
such that $0\leq a_l< m_l, 0\leq a_{ij}<(m_i,m_j), 0\leq a_{rst}<(m_r,m_s,m_t)$ for $1\leq l \leq N, 1\leq i<j\leq  N,1\leq r<s<t\leq N$ where $a_{ij}$ and
$a_{rst}$ are listed according to the lexicographic order. For brevity, the sequence of form (4.2) will be denoted by $\overline{a}$ in the following. Now for
 any $\overline{a}\in A$, define a map $\Phi_{\overline{a}}:G\times G\times G\longrightarrow \k^*$ by
\begin{equation}
\Phi_{\overline{a}}(e_1^{i_1}\cdots e_N^{i_N},e_1^{j_1}\cdots e_N^{j_N},e_1^{k_1}\cdots e_N^{k_N})
=\prod_{l=1}^N\zeta_{m_l}^{a_li_l[\frac{j_l+k_l}{m_l}]}\prod_{1\leq s<t\leq N}\zeta_{m_t}^{a_{st}i_t[\frac{j_s+k_s}{m_s}]}\prod_{1\leq r< s<t\leq N}
\zeta_{(m_r,m_s,m_t)}^{a_{rst}k_rj_si_t}
\end{equation}
where $[x]$ means the integer part of $x$ and $\zeta_{m_l}$ is an $m_l$-th primitive root of unity. According to \cite{bgrc1,bgrc2}, $\{\Phi_{\overline{a}}|
\overline{a}\in A\}$ is a complete set of normalized 3-cocycles on $G$ up to cohomology.

\subsection{Quasi-characters}
\begin{definition}
A function $\chi:G\To \k^*$ is called a quasi-character associated to a 2-cocycle $\omega$ on $G$ if
\begin{equation}
\chi(f)\chi(g)=\omega(f,g)\chi(fg),\ \ \ \chi(1)=1
\end{equation}
for all $f,g\in G.$
\end{definition}

It is clear that there is a quasi-character associated to $\omega$ if and only if $\omega$ is symmetric. Recall that in Subsection 2.2, a set of
 2-cocycles $\widetilde{\Phi}_g$ on $G$ were defined for a given 3-cocycle $\Phi.$ For later applications we need the following definition associated to such a set.

\begin{definition}
A series of quasi-characters $\{\chi_1,\cdots,\chi_n\}$ of $G$ is called admissible with respect to $\Phi$ if $\chi_i$ is associated to
 $\widetilde{\Phi}_{g_i}$ with $g_i\neq 1$ for $1\leq i\leq n$ such that $G=\langle g_1,\cdots, g_n\rangle$ and
\begin{equation}
\chi_i(g_j)\chi_j(g_i)=1,\ \ \ \chi_i(g_i)\neq 1
\end{equation}
for all $1\leq i,j \leq n.$
\end{definition}

Note that, in the preceding definition, the $g_i$'s are not necessarily distinct.

\begin{proposition}
 If there is a series of quasi-characters $\{\chi_1,\cdots,\chi_n\}$ of $G$ such that $\chi_i$ is associated to $\widetilde{\Phi}_{g_i}$ for each $i$
 and $G=\langle g_1,\cdots,\ g_n\rangle,$  then up to cohomology $\Phi$ must be of the form
\begin{equation}
\Phi(e_1^{i_1}\cdots e_N^{i_N},e_1^{j_1}\cdots e_N^{j_N},e_1^{k_1}\cdots e_N^{k_N})
=\prod_{l=1}^N\zeta_l^{a_li_l[\frac{j_l+k_l}{m_l}]}\prod_{1\leq s<t\leq N}\zeta_{m_t}^{a_{st}i_t[\frac{j_s+k_s}{m_s}]}.
\end{equation}

\end{proposition}
\begin{proof}
It is known that $\Phi$ can be chosen as the form of formula (4.3). We need to prove that the assumptions in the proposition lead to $a_{rst}=0$ for all
$1\leq r<s<t\leq N.$

As $G=\langle g_1,\cdots,g_n\rangle,$ one has $e_r=g_1^{k_1}\cdots g_n^{k_n}$ for some $k_1<|g_1|,\cdots, k_n<|g_n|.$ Here and below, $|g|$ denotes the order
of $g$ in a group. Write $g_i=e_1^{c_{i1}}\cdots \ e_N^{c_{iN}}$ for $1\leq i\leq n,$ then we have
\begin{equation*}
\sum_{i=1}^n k_ic_{il} \equiv \begin{cases} 0 \ \mod m_l, & l\neq r; \\ 1 \ \mod m_r, & l=r.
\end{cases}
\end{equation*}
By (4.3) we can see that
\begin{equation*}
\Phi(g_i,e_s,e_t)
=\prod_{1\leq j<s}\zeta_{(m_j,m_s,m_t)}^{c_{ij}a_{jst}}.
\end{equation*}
It follows that $$\prod_{i=1}^{n}\Phi(g_i,e_s,e_t)^{k_i}= \prod_{1\leq j<s}\zeta_{(m_j,m_s,m_t)}^{a_{jst}(\sum_{i=1}^nk_ic_{ij})}=\zeta_{(m_r,m_s,m_t)}^{a_{rst}}.$$

On the other hand, by direct computation we have $\widetilde{\Phi}_{g_i}(e_s,e_t)=1$ if $s <t.$ Since there is a quasi-character associated to
the 2-cocycle $\widetilde{\Phi}_{g_i},$ it is symmetric. This leads to $$\widetilde{\Phi}_{g_i}(e_t,e_s)=\widetilde{\Phi}_{g_i}(e_s,e_t)=1.$$
This implies
\begin{equation*}
\Phi(g_i,e_s,e_t)=\frac{\Phi(e_s,g_i,e_t)}{\Phi(e_s,e_t,g_i)}
=\prod_{s<p<t}\zeta_{(m_s,m_p,m_t)}^{c_{ip}a_{spt}}\bigg[\prod_{t<q<N}\zeta_{(m_s,m_t,m_q)}^{c_{iq}a_{stq}}\bigg]^{-1}.
\end{equation*}
Then
\begin{equation*}
\begin{split}
\prod_{i=1}^{n}\Phi(g_i,e_s,e_t)^{k_i}=&\prod_{i=1}^n\bigg\{\prod_{s<p<t}\zeta_{(m_s,m_p,m_t)}^{c_{ip}a_{spt}}\bigg[\prod_{t<q<N}
\zeta_{(m_s,m_t,m_q)}^{c_{iq}a_{stq}}\bigg]^{-1}\bigg\}^{k_i}\\
=&\prod_{s<p<t}\zeta_{(m_s,m_p,m_t)}^{a_{spt}(\sum_{i=1}^n k_ic_{ip})}\bigg[\prod_{t<q<N}\zeta_{(m_s,m_t,m_q)}^{a_{stq}(\sum_{i=1}^n k_ic_{iq})}\bigg]^{-1}\\
=&1.
\end{split}
\end{equation*}
So we get $\zeta_{(m_r,m_s,m_t)}^{a_{rst}}=1$ which implies $a_{rst}=0$ since $0\leq a_{rst}<(m_r,m_s,m_t).$
\end{proof}

From (4.6), it is easy to see that
\begin{eqnarray}
&\Phi(e,f,g)=\Phi(e,g,f),\\
&\Phi(ef,g,h)=\Phi(e,g,h)\Phi(f,g,h)
\end{eqnarray}
for any $e,f,g,h\in G.$ Hence we have
\begin{equation}
\widetilde{\Phi}_{e}(f,g)=\frac{\Phi(e,f,g)\Phi(f,g,e)}{\Phi(f,e,g)}=\Phi(e,f,g).
\end{equation}

\subsection{Braided linear spaces}
The aim of this subsection is to give a classification of braided linear spaces in $_{\k G}^{\k G}\mathcal{Y}\mathcal{D}^{\Phi}.$ Without loss of the generality,
 we may assume that if a braided linear space $\mathcal{S}$ in $_{\k G}^{\k G}\mathcal{Y}\mathcal{D}^{\Phi}$ is generated by a set $\{X_1,\cdots,X_n\}$ of
 primitive elements
and the $G$-degree of $X_i$ is $g_i$ for $1\leq i\leq n,$ then $\langle g_i,\cdots,g_n\rangle=G.$ In fact, if $\langle g_i,\cdots,g_n\rangle\lneqq G,$
let $G'=\langle g_i,\cdots,g_n\rangle,$ then $\mathcal{S}$ is actually a braided linear space in $_{\k G'}^{\k G'}\mathcal{Y}\mathcal{D}^{\Phi'}$ such
 that $\langle g_i,\cdots,g_n\rangle=G',$ where $\Phi'=\Phi|_{G'}.$

\begin{lemma}
Keep the assumptions of the braided linear space $\mathcal{S}$ in $_{\k G}^{\k G}\mathcal{Y}\mathcal{D}^{\Phi}.$ Then $\k X_i$ is a $1$-dimensional
$(G,\widetilde{\Phi}_{g_i})$-representation.
\end{lemma}

\begin{proof}
Let $V=\oplus_{1\leq i\leq n}\k X_i.$ As mentioned in Subsection 2.4, $V$ is an object in $_{\k G}^{\k G}\mathcal{Y}\mathcal{D}^{\Phi}.$
For $i\neq j,$ we have by the defining relations of $\mathcal{S}$
\begin{equation*}
\begin{split}
&\b(X_iX_j-q_{j,i}X_jX_i)\\
=&(X_iX_j-q_{j,i}X_jX_i)\otimes 1+(g_i\triangleright X_j-q_{j,i}X_j)\otimes X_i+(X_i-q_{j,i}g_j\triangleright X_i)\otimes X_j+1\otimes(X_iX_j-q_{j,i}X_jX_i)\\
=&0.
\end{split}
\end{equation*}
So we have \begin{equation}g_i\triangleright X_j=q_{j,i}X_j,\ \ q_{j,i}g_j\triangleright X_i= X_i. \end{equation}

It remains to consider $g_i\triangleright X_i.$ If there are $g_j=g_i$ for some $j \neq i,$ then we have $g_i\triangleright X_i=g_j\triangleright X_i=q_{i,j}X_i.$
Then the $G$-action on $\k X_i$ is stable and it is a $(G,\widetilde{\Phi}_{g_i})$-representation. Otherwise, $g_i\neq g_j$ for any $j \neq i,$ then
 it is obvious that $\k X_i=V_{g_i}$ and so $\k X_i$ is a $(G,\widetilde{\Phi}_{g_i})$-representation as well by Proposition 2.4.
\end{proof}

\begin{lemma}
Suppose that $\k X$ is a $(G,\widetilde{\Phi}_h)$-representation with action $g\triangleright X=q X.$ Then we have
\begin{equation}
g^k\triangleright X=\prod_{j=1}^{k-1}\Phi(h,g^j,g)^{-1}q^kX.
\end{equation}
\end{lemma}
\begin{proof}
This is obvious by induction on $k.$
\end{proof}

\begin{corollary}
Keep the assumptions of the previous lemma. Then
\begin{equation}
q^{ln}=\prod_{j=1}^{ln-1}\Phi(h,g^j,g),
\end{equation}
where $n=|g|.$
\end{corollary}

\begin{lemma}
Let $V\in _{\k G}^{\k G}\mathcal{Y}\mathcal{D}^{\Phi}$ where $\Phi$ is of the form of $\mathrm{(4.6)}.$ If $X \in V_g$ is nonzero satisfying
$g\triangleright X=qX,$ then we have in $T(V)$
\begin{equation}
\Delta(X^{\overrightarrow{m}})=\sum_{i=0}^{m}{m\choose i}_qX^{\overrightarrow{i}}\widetilde{\otimes }X^{\overrightarrow{m-i}}
\end{equation}
where $X^{\overrightarrow{i}}\widetilde{\otimes }X^{\overrightarrow{m-i}}=\prod_{j=1}^{m-1-i}\Phi(g^i,g^j,g)^{-1}
X^{\overrightarrow{i}}\otimes X^{\overrightarrow{m-i}}.$
\end{lemma}
\begin{proof}
We prove this lemma by induction on $m.$

If $m=1,$ the identity is obvious.

Suppose the identity is correct for $m,$ i.e.,
$\Delta(X^{\overrightarrow{m}})=\sum_{i=0}^{m}{m\choose i}_qX^{\overrightarrow{i}}\widetilde{\otimes }X^{\overrightarrow{m-i}}.$ Now consider
\begin{equation}
\b(X^{\overrightarrow{m+1}})=\b(X^{\overrightarrow{m}})\b(X)
=\sum_{i=0}^{m}{m\choose i}_q(X^{\overrightarrow{i}}\widetilde{\otimes }X^{\overrightarrow{m-i}})(X\otimes 1+1\otimes X).
\end{equation}

We have
\begin{equation*}
\begin{split}
{m\choose i}_q(X^{\overrightarrow{i}}\widetilde{\otimes }X^{\overrightarrow{m-i}})(1\otimes X)
=&{m\choose i}_q\prod_{j=1}^{m-1-i}\Phi(g^i,g^j,g)^{-1}
(X^{\overrightarrow{i}}\otimes X^{\overrightarrow{m-i}})(1\otimes X)\\
=&{m\choose i}_q\prod_{j=1}^{m-1-i}\Phi(g^i,g^j,g)^{-1}\Phi(g^{i},g^{(m-i)},g)^{-1}X^{\overrightarrow{i}}\otimes X^{\overrightarrow{m+1-i}}\\
=&{m\choose i}_q\prod_{j=1}^{m-i}\Phi(g^i,g^j,g)^{-1}X^{\overrightarrow{i}}\otimes X^{\overrightarrow{m+1-i}}\\
=&{m\choose i}_qX^{\overrightarrow{i}}\widetilde{\otimes }X^{\overrightarrow{m+1-i}}
\end{split}
\end{equation*}
and

\begin{equation*}
\begin{split}
&{m\choose i-1}_q(X^{\overrightarrow{i-1}}\widetilde{\otimes }X^{\overrightarrow{m+1-i}})(X\otimes 1)\\
=&{m\choose i-1}_q\prod_{j=1}^{m-i}\Phi(g^{i-1},g^j,g)^{-1}(X^{\overrightarrow{i-1}}\otimes X^{\overrightarrow{m+1-i}})(X\otimes 1) \\
=&{m\choose i-1}_q\prod_{j=1}^{m-i}\Phi(g^{i-1},g^j,g)^{-1}\frac{\Phi(g^{(i-1)},g^{(m+1-i)},g)}{\Phi(g^{(i-1)},g,g^{(m+1-i)})}
[X^{\overrightarrow{i-1}}(g^{(m+1-i)}\triangleright X)]\otimes X^{\overrightarrow{m+1-i}}\\
=&{m\choose i-1}_q\prod_{j=1}^{m-i}\Phi(g^{i-1},g^j,g)^{-1}\prod_{j=1}^{m-i}\Phi(g,g^j,g)^{-1}q^{(m+1-i)}
X^{\overrightarrow{i}}\otimes X^{\overrightarrow{m+1-i}}\\
=&{m\choose i-1}_q\prod_{j=1}^{m-i}\Phi(g^{i},g^j,g)^{-1}q^{(m+1-i)}
X^{\overrightarrow{i}}\otimes X^{\overrightarrow{m+1-i}}\\
=&{m\choose i-1}_qq^{m+1-i}X^{\overrightarrow{i}}\widetilde{\otimes }X^{\overrightarrow{m+1-i}}.
\end{split}
\end{equation*}
Note that the third equality follows from (4.7) and Lemma 4.5, and the fourth follows from (4.8).

By the previous two identities, we obtain
\begin{equation*}
{m\choose i}_q(X^{\overrightarrow{i}}\widetilde{\otimes }X^{\overrightarrow{m-i}})(1\otimes X)+{m\choose i-1}_q
(X^{\overrightarrow{i-1}}\widetilde{\otimes }X^{\overrightarrow{m+1-i}})(X\otimes 1)={m+1\choose i}_qX^{\overrightarrow{i}}\widetilde{\otimes}X^
{\overrightarrow{m+1-i}},
\end{equation*}
 hence
\begin{equation*}
\b(X^{\overrightarrow{m+1}})
=\sum_{i=0}^{m+1}{m+1 \choose i}_q(X^{\overrightarrow{i}}\widetilde{\otimes }X^{\overrightarrow{m+1-i}}).
\end{equation*}

The lemma is proved.
\end{proof}

\begin{corollary}
Suppose that $X$ is a primitive element of $\S(V)$ of $G$-degree $g$ such that $g\triangleright X=qX,$ then  $q\neq 1$ and
 the  nilpotent order of $X$ is $|q|,$ i.e., $|q|$ is the minimal positive integer $m$ such that $X^{\overrightarrow{m}}=0.$
\end{corollary}

\begin{proof}
If $q=1,$ then by Lemma 4.7 we have in $T(V)$
\begin{equation*}
\Delta(X^{\overrightarrow{m}})=\sum_{i=0}^{m}{m\choose i}X^{\overrightarrow{i}}\widetilde{\otimes }X^{\overrightarrow{m-i}}.
\end{equation*}
This implies that in $\S(V)$ if  $X^{\overrightarrow{m}}\neq 0$ then $X^{\overrightarrow{m+1}}\neq 0.$ It follows that $\{ X^{\overrightarrow{m}} | m \neq 0 \}$
 is a linearly independent set of $\S(V).$ This is absurd as $\S(V)$ is assumed to be finite dimensional.

So $q \neq 1,$ and clearly $q$ is a root of a unity. Let $l=|q|.$ Then by (4.13) we can see that $l$ is the unique number greater than $1$ such that
\begin{equation*}
\Delta(X^{\overrightarrow{l}})=X^{\overrightarrow{l}}\otimes1 +1\otimes X^{\overrightarrow{l}},
\end{equation*}
i.e., $X^{\overrightarrow{l}}$ is a primitive element of $T(V)$ of degree $l > 1,$ hence must be $0$ in $\S(V).$
\end{proof}

Now we are ready to give one of the main results of this paper.

\begin{theorem}
The set of braided linear spaces of rank $n$ in $_{\k G}^{\k G}\mathcal{Y}\mathcal{D}^{\Phi}$ is in one-to-one correspondence with the set of admissible series of quasi-characters $\{\chi_1,\cdots,\chi_n\}$ of $G$ with respect to $\Phi.$ More precisely, if an admissible series of quasi-characters $\{\chi_1,\cdots,\chi_n\}$
of $G$ with respect to $\Phi$ is given, the corresponding braided linear space $\mathcal{S}$ in $_{\k G}^{\k G}\mathcal{Y}\mathcal{D}^{\Phi}$ can be presented by
\begin{eqnarray}
&X_i^{\overrightarrow{N_i}}=0\ where \ N_i=|\chi_i(g_i)|, \ 1 \le i \le n,\\
& X_iX_j=\chi_j(g_i)X_jX_i, \ 1\leq i\neq j\leq n
\end{eqnarray}
and the coalgebra structure is determined by $\b(X_i)=X_i\otimes 1+ 1\otimes X_i$ and $\varepsilon(X_i)=0$ for all $i.$
\end{theorem}
\begin{proof}
Suppose that $\S$ is a braided linear space of rank $n$ in $_{\k G}^{\k G}\mathcal{Y}\mathcal{D}^{\Phi}.$ Then by Lemma 4.4, it is not hard to show that  $\S(1)$
provides an admissible series of quasi-characters of $G$ with respect to $\Phi.$

Conversely, given an admissible series of quasi-characters $\{\chi_1,\cdots,\chi_n\}$ of $G$ with respect to $\Phi,$ then there are one-dimensional
$(G,\widetilde{\Phi}_{g_i})$-representations $\k X_i$ $(1\leq i\leq n).$ Let $V=\oplus_{1\leq i\leq n}\k X_i$ and by Proposition 2.4 $V$ becomes an object
 in $_{\k G}^{\k G}\mathcal{Y}\mathcal{D}^{\Phi}$ if the $G$-degree of $X_i$ is set to be $g_i.$ Consider the tensor algebra $T(V)$ in the category. As the
  elements in $$\{X_i^{N_i},X_iX_j-\chi_j(g_i)X_jX_i|N_i=|\chi_i(g_i)|,1\leq i\leq n\}$$ are primitive in $T(V),$ the ideal $I$ generated by this set is a
  braided Hopf ideal of $T(V)$ in the category $_{\k G}^{\k G}\mathcal{Y}\mathcal{D}^{\Phi}.$ Now the quotient braided Hopf algebra $\S(V)=T(V)/I$ is the desired braided linear space.
\end{proof}

\begin{remark}
Keep the notations and assumptions of the above theorem. Then $\S(V)$ is isomorphic to the Nichols algebra $\mathcal{B}(V)$ of $V$ in
$_{\k G}^{\k G}\mathcal{Y}\mathcal{D}^{\Phi}.$ This is clear by Proposition 3.3, Theorem 3.6 and the fact that $[\S(V) \# \k G]^{\operatorname{coinv} \k G}=\S(V).$
In other words, $\S(V)$ is a commutative Nichols algebra in the braided tensor category $_{\k G}^{\k G}\mathcal{Y}\mathcal{D}^{\Phi}.$
\end{remark}

\section{Finite quasi-quantum linear spaces}
In this section we give a classification of finite quasi-quantum linear spaces. As mentioned earlier in the introduction, this amounts to a classification of
finite-dimensional graded pointed Majid algebras generated by an abelian group and a set of skew-primitive elements which are mutually quasi-commutative.

Let $M$ be a finite-dimensional graded pointed Majid algebra generated by an abelian group $G$ and a set of quasi-commutative skew-primitive elements $\{ X_i | 1 \le i \le n \}.$ Assume further that $\Delta(X_i)=X_i\otimes 1+g_i\otimes X_i$ for $1\leq i\leq n$ and that $G=\langle g_1,\cdots,g_n\rangle$
throughout this section. Then $M_0=\k G$ and the associated coinvariant subalgebra $R$ is a braided linear space in $_{\k G}^{\k G}\mathcal{Y}\mathcal{D}^{\Phi}.$
In fact, if we let $V=\oplus_{1\leq i\leq n}\k X_i$ as in Section 4, then by Remark 4.10 $R \cong \S(V)$ as braided Hopf algebras in
$_{\k G}^{\k G}\mathcal{Y}\mathcal{D}^{\Phi}.$  We apply the results of Sections 3 and 4 to determine $M.$ First we provide the explicit set of admissible sets
of quasi-characters of $G$ with respect to a 3-cocycle $\Phi$ on $G,$ and then we carry out the bosonization procedure for the corresponding braided linear spaces
in the braided tensor category $_{\k G}^{\k G}\mathcal{Y}\mathcal{D}^{\Phi}.$

As in Section 4, we assume that $G=\Z_{m_1} \times \cdots \times \Z_{m_N}$ and $e_i$ is a fixed generator of $\Z_{m_i}$ for $1\leq i\leq N.$ We also keep the
notations of Section 4.

\subsection{Quasi-characters associated to $\widetilde{\Phi}_g$}
In view of Proposition 4.3, in this section we only need to consider 3-cocycle $\Phi$ of the following form
\[  \Phi(e_1^{i_1}\cdots e_N^{i_N},e_1^{j_1}\cdots e_N^{j_N},e_1^{k_1}\cdots e_N^{k_N}) =\prod_{l=1}^N\zeta_l^{a_li_l[\frac{j_l+k_l}{m_l}]}
\prod_{1\leq s<t\leq N}\zeta_{m_t}^{a_{st}i_t[\frac{j_s+k_s}{m_s}]}.  \] In this case, $\widetilde{\Phi}_g(e,f)=\Phi(g,e,f)$ and it is a symmetric 2-cocycle.
 Moreover, there are quasi-characters associated to $\widetilde{\Phi}_g$ for any $g\in G.$ The following lemma gives an explicit presentation of the
 quasi-characters associated to $\widetilde{\Phi}_g.$

\begin{lemma}
Let $\Phi$ be a 3-cocycle on $G$ of the form $(4.6).$ Then $\chi$ is a quasi-character of $G$ associated to $\widetilde{\Phi}_g$ for
$g=e_1^{i_1}\cdots e_N^{i_N}$ if and only if
\begin{equation}
\chi(e_l)=\big[\zeta_{m_l}^{a_li_l}\prod_{l<t\leq N}\zeta_{m_t}^{a_{lt}i_t}\big]^{\frac{1}{m_l}}
\end{equation}
for $1\leq l\leq N.$
\end{lemma}

\begin{proof}
By (4.4), we have
\begin{equation*}
\chi(e_l^a)\chi(e_l^b)=\widetilde{\Phi}_g(e_l^a,e_l^b)\chi(e_l^{a+b})=\Phi(g,e_l^a,e_l^b)\chi(e_l^{a+b})=\zeta_{m_l}^{a_li_l[\frac{a+b}{m_l}]}
\prod_{l<t\leq N}\zeta_{m_t}^{a_{lt}i_t[\frac{a+b}{m_l}]}\chi(e_l^{a+b}).
\end{equation*}
From this identity it is easy to see that $\chi(e_i^a)=\chi(e_i)^a$ for $a<m_l$ and that
\begin{equation*}
\chi(e_l^{m_l})=\zeta_{m_l}^{-a_li_l}\prod_{l<t\leq N}\zeta_{m_t}^{-a_{lt}i_t}\chi(e_l^{m_l-1})\chi_(e_l)=
\zeta_{m_l}^{-a_li_l}\prod_{l<t\leq N}\zeta_{m_t}^{-a_{lt}i_t}\chi(e_l)^{m_l}=1,
\end{equation*}
then $\chi(e_l)=\big[\zeta_{m_l}^{a_li_l}\prod_{l<t\leq N}\zeta_{m_t}^{a_{lt}i_t}\big]^{\frac{1}{m_l}}$ for $1\leq l\leq N.$

Conversely, assume $\chi(e_l)=\big[\zeta_{m_l}^{a_li_l}\prod_{l<t\leq N}\zeta_{m_t}^{a_{lt}i_t}\big]^{\frac{1}{m_l}}$ for $1\leq l\leq N.$ Define
\begin{equation*}
\chi(e_1^{r_1}\cdots e_N^{s_N})=\chi(e_1)^{s_1}\cdots\chi(e_N)^{s_N}.
\end{equation*}
Then by a routine calculation one can verify that $\chi$ is a quasi-character associated to $\widetilde{\Phi}_g.$
\end{proof}

\subsection{Admissible series of quasi-characters}
Let $\{\chi_1,\cdots,\chi_n\}$ be an admissible series of quasi-characters of $G$ with respect to $\Phi$ where $\chi_i$ is a quasi-character associated to
 $\widetilde{\Phi}_{g_i}$ for $1\leq i\leq n.$ Suppose that $g_i=e_1^{\alpha_{i1}}\cdots e_N^{\alpha_{iN}}$ for $1\leq i\leq n.$ Then we get an $n \times N$
  matrix $(\alpha_{ij})$ with integer entries.

\begin{lemma}
Suppose $A=(\alpha_{ij})$ is an $n\times N$ matrix as above. Then there is an admissible series of quasi-characters $\{\chi_1,\cdots,\chi_n\}$ of $G$ with
respect to $\Phi$  with $\chi_i$ associated to $\widetilde{\Phi}_{g_i}$ if and only if
\begin{equation}
\prod_{l=1}^{N}\bigg[\zeta_{m_l}^{\frac{2a_l\alpha_{il}\alpha_{jl}}{m_l}}\prod_{l<t\leq N}
\zeta_{m_t}^{\frac{a_{lt}(\alpha_{it}\alpha_{jl}+\alpha_{jt}\alpha_{il})}{m_l}}\bigg]=1 , \ \ 1\leq i\neq j\leq n, \ \ \mathrm{and}
\end{equation}
\begin{equation}
\prod_{l=1}^{N}\bigg[\zeta_{m_l}^{\frac{a_l\alpha_{il}^2}{m_l}}\prod_{l<t\leq N}
\zeta_{m_t}^{\frac{a_{lt}\alpha_{it}\alpha_{il}}{m_l}}\bigg]\neq 1 , \ \ 1\leq i\leq n.
\end{equation}
\end{lemma}
\begin{proof}
Assume that $\{\chi_1,\cdots,\chi_n\}$ is an admissible series of quasi-characters of $G$ with respect to $\Phi$ with $\chi_i$ associated to
$\widetilde{\Phi}_{g_i}.$ Then by Lemma 5.1, we have
$$\chi_i(e_l)=\big[\zeta_{m_l}^{a_l\alpha_{il}}\prod_{l<t\leq N}\zeta_{m_t}^{a_{lt}\alpha_{it}}\big]^{\frac{1}{m_l}},$$
hence
\begin{equation*}
\chi_i(g_j)=\prod_{l=1}^{N}\bigg[\zeta_{m_l}^{\frac{a_l\alpha_{il}\alpha_{jl}}{m_l}}\prod_{l<t\leq N}
\zeta_{m_t}^{\frac{a_{lt}\alpha_{it}\alpha_{jl}}{m_l}}\bigg].
\end{equation*}
Now (5.2) follows from the condition $\chi_i(g_j)\chi_j(g_i)=1$ for $1\leq i\neq j\leq n,$ and (5.3) follows from the condition
$\chi_i(g_i)\neq 1$ for $1\leq i\leq n.$

Conversely, suppose (5.2) and (5.3) hold. Define quasi-characters $\chi_i$ associated to $\widetilde{\Phi}_{g_i}$ for $1\leq i\leq n$ by
$$\chi_i(e_l)=\big[\zeta_{m_l}^{a_l\alpha_{il}}\prod_{l<t\leq N}\zeta_{m_t}^{a_{lt}\alpha_{it}}\big]^{\frac{1}{m_l}}, \ \ 1\leq l\leq N.$$
By a direct verification, one can show that $\{\chi_1,\cdots,\chi_n\}$ is an admissible series of quasi-characters of $G$ with respect to $\Phi.$
\end{proof}

\subsection{Classification results}
By $\mathcal{A}(G,\Phi)$ we denote the set of those $n \times N$ integer matrices $(\alpha_{ij})$ satisfying the conditions given in Subsection 5.2. Finally
 we are ready to give the main classification result.

\begin{theorem}
Given an $n \times N$ matrix $(\alpha_{ij})$ in $\mathcal{A}(G,\Phi),$ we can define a finite-dimensional graded pointed Majid algebra $M$ generated by $G$
and a set
$\{X_1,\cdots,X_n\}$ of skew-primitive elements subject to relations
$$e_iX_j=\big[\zeta_{m_i}^{a_i\alpha_{ji}}\prod_{i<t\leq N}\zeta_{m_t}^{a_{it}\alpha_{jt}}\big]^{\frac{1}{m_i}}X_je_i,\ \ 1\leq i\leq N,1\leq j\leq n,$$
$$X_iX_j=\prod_{l=1}^{N}\bigg[\zeta_{m_l}^{\frac{a_l\alpha_{il}\alpha_{jl}}{m_l}}\prod_{l<t\leq N}
\zeta_{m_t}^{\frac{a_{lt}\alpha_{jt}\alpha_{il}}{m_l}}\bigg] X_jX_i,\ \ \  1\leq i\neq j\leq n, $$
$$ X_i^{\overrightarrow{N_i}}=0, \ \  N_i=\bigg|\prod_{l=1}^{N}\bigg[\zeta_{m_l}^{\frac{a_l\alpha_{il}^2}{m_l}}\prod_{l<t\leq N}
\zeta_{m_t}^{\frac{a_{lt}\alpha_{it}\alpha_{il}}{m_l}}\bigg]\bigg|,\ \ 1\leq i\leq n.$$
The coproduct of $M$ is determined by $\b(g)=g\otimes g$ for any $g\in G$ and  $\b(X_i)=X_i\otimes 1+g_i\otimes X_i$ where
$g_i=e_1^{\alpha_{i1}}\cdots e_N^{\alpha_{iN}}$ for $1\leq i\leq n.$ The associator is obtained by extending $\Phi$ as in Section 3. Conversely, any
 finite-dimensional graded pointed Majid algebra generated by $G$ and a set $\{X_1,\cdots,X_n\}$ of skew-primitive elements satisfying the quasi-commutative
  condition is twist equivalent to one of the Majid algebras defined above.
\end{theorem}

\begin{proof} Direct consequence of Proposition 3.3 and Theorem 4.9. \end{proof}

\subsection{Examples of pointed Majid algebras over $\Z_2\times \Z_2\times \Z_2$}
We conclude the paper by an explicit classification of the finite quasi-quantum linear spaces over the group $G=\Z_2\times \Z_2\times \Z_2.$ On the one hand, these examples are new and are very interesting in their own right. On the other hand, the method of computations here provides a typical model for any finite abelian group.

Let $M$ be a finite-dimensional graded pointed Majid algebra generated by $G$ and a set of quasi-commutative skew-primitive elements $\{ X_i \mid 1 \le i \le N \}.$ As before denote the degree of $X_i$ by $x_i.$ We also assume that the set $\{ x_i \mid 1\leq i\leq N \}$ generates the group $G.$ Hence the pointed Majid algebras $M$ is of rank $\ge 3,$ i.e., $N \geq 3.$ Besides, as $\langle x_1,\cdots,x_N \rangle=G,$ so three of them, say $x_1,x_2,x_3,$ generate $G.$ The following lemma is obvious.

\begin{lemma}
With the above assumptions, we have $G=\langle x_1\rangle \times \langle x_2 \rangle \times \langle x_3 \rangle.$
\end{lemma}

Then by Proposition 4.3, the $3$-cocycles on $G$ can be chosen in the following form
\begin{equation}
\Phi(x_1^{i_1}x_2^{i_2}x_3^{i_3},x_1^{j_1}x_2^{j_2}x_3^{j_3},x_1^{k_1}x_2^{k_2}x_3^{k_3})
=\prod_{1 \le l \le 3}{(-1)}^{a_li_l[\frac{j_l+k_l}{2}]}\prod_{1\leq s<t\leq 3}{(-1)}^{a_{st}i_t[\frac{j_s+k_s}2]},
\end{equation}
where $a_l, a_{st} \in \{0,1\}.$ It turns out that as the associator of $M$ the 3-cocycles $\Phi$ should be further restricted as follows.

\begin{proposition}
Keep the above assumptions on $M.$ Then its associator $\Phi$ must be of the following form \begin{equation}\Phi(x_1^{i_1}x_2^{i_2}x_3^{i_3},x_1^{j_1}x_2^{j_2}x_3^{j_3},x_1^{k_1}x_2^{k_2}x_3^{k_3})
=\prod_{l=1}^3{(-1)}^{a_li_l[\frac{j_l+k_l}{2}]} \end{equation} with each $a_l\in \{0,1\}.$
\end{proposition}

\begin{proof}
Let $V=\k\{X_1,\cdots,X_N\},$ then $V\in\ _{\k G}^{\k G}\mathcal{Y}\mathcal{D}^{\Phi}$ and $M=B(V)\#\k G$ by Theorem 3.6. According to Theorem 4.9, there is an admissible series of quasi-characters $\{\chi_1,\cdots,\chi_N\}$ with respect to $\Phi$ which corresponds to $B(V),$ and each
 $\chi_i$ is associated to $\Phi_{x_i}$ for all $1\leq i\leq N.$ Then by (5.2), we get $(-1)^{\frac{a_{st}}{2}}=1$ for all $1 \leq s< t\leq 3.$ Note that here we have used the facts
 $\alpha_{ij}=\delta_{ij}$ for all $1\leq i,j\leq 3.$ So we get $a_{st}=0$ for all $1\leq s < t\leq 3$ and the claim follows.
\end{proof}

Now there are possibly seven nontrivial associators for finite quasi-quantum linear spaces over $G.$ For each case we can compute the admissible series of quasi-characters in the same manner.
In the following, we will take $$\Phi(x_1^{i_1}x_2^{i_2}x_3^{i_3},x_1^{j_1}x_2^{j_2}x_3^{j_3},x_1^{k_1}x_2^{k_2}x_3^{k_3})
=\prod_{l=1}^3{(-1)}^{i_l[\frac{j_l+k_l}{2}]}$$ for example. In other words, we fix the $\Phi$ as in (5.5) with $a_1=a_2=a_3=1.$ By Lemma 5.1, we have
\begin{equation}
\chi_1(x_1)=\pm \mathrm{i},\ \  \chi_1(x_2)=\pm 1,\ \  \chi_1(x_3)=\pm 1.
\end{equation}
Here and below, $\mathrm{i}$ stands for $\sqrt{-1}.$ Similarly we have
\begin{eqnarray}
&\chi_2(x_1)=\pm 1,\ \  \chi_2(x_2)=\pm \mathrm{i},\ \  \chi_2(x_3)=\pm 1;\\
&\chi_3(x_1)=\pm 1,\ \  \chi_3(x_2)=\pm 1,\ \  \chi_3(x_3)=\pm \mathrm{i}.
\end{eqnarray}
Also by taking (4.5) into consideration, we obtain
\begin{equation}
\chi_1(x_2)=\chi_2(x_1)=\pm 1,\ \chi_1(x_3)=\chi_3(x_1)=\pm 1,\ \chi_2(x_3)=\chi_3(x_2)=\pm 1.
\end{equation}
From (5.6-5.9), we can easily get all cases of rank $3$ admissible quasi-characters $\{\chi_1,\chi_2,\chi_3\}$ of $G$ with respect to the given $\Phi.$
If $N>3,$ then according to Definition 4.2 we have
\begin{equation}
\chi_j(x_i)=\chi_i^{-1}(x_j)
\end{equation} for all $1\leq i\leq 3$ and $j>3.$ As $\chi_1,\chi_2,\chi_3$ are fixed, thus any other character $\chi_j \ (j>3)$ which is compatible with $\{\chi_1, \chi_2, \chi_3\},$ i.e., the conditions of Definition 4.2 hold, is uniquely determined by $x_j$ due to (5.10) by noting that $x_1,x_2,x_3$ generate $G$ and that $\chi_j$ is multiplicative. Since $x_i\neq 1,$ so there are possibly at most
$7$ other classes of quasi-characters than $\chi_1,\chi_2,\chi_3$ which obviously correspond to the set of non-identity elements of $G.$ For the convenience of exposition, we make a convention for the notations of the quasi-characters $\chi_j \ (j \ge 4).$ Let $\chi'_i$ denote the quasi-character in the series $\{\chi_j\}_{j \ge 4}$ corresponding to $x_i$ for $1\leq i\leq 3,$ $\chi_{ij}$ the quasi-character corresponding to $x_ix_j$ for $1\leq i<j\leq 3,$
and $\chi_{123}$ the quasi-character corresponding to $x_1x_2x_3.$ With the above notations and assumptions, we have

\begin{proposition}
Any admissible series of quasi-characters of $G$ with respect to the fixed $\Phi$ must be one of the following:
\begin{itemize}
\item[(1)] $\{\chi_1,\cdots,\chi_N\}$ with $3\leq N\leq 6,$ $\chi_4,\cdots,\chi_N$ are distinct, and for each $j \ge 4, \ \chi_j \in \{\chi'_1,\chi'_2,\chi'_3\};$
\item[(2)] $\{\chi_1,\cdots,\chi_N\}$ with $N\geq 5,$ $\chi_4=\chi'_k$ for some $1\leq k\leq 3,$  $\chi_i(x_i)=\chi_j(x_j)$ for $\{i,j\}=\{1,2,3\} \setminus \{k\},$ and $\chi_5=\cdots=\chi_N=\chi_{ij};$
 \item[(3)] $\{\chi_1,\cdots,\chi_N\}$ with $N\geq 4,$ $\chi_4=\cdots=\chi_N =\chi_{ij}$ for some $1\leq i<j\leq 3,$ and $\chi_i(x_i)=\chi_j(x_j);$
\item[(4)] $\{\chi_1,\cdots,\chi_N\}$ with $N\geq 4,$ $\chi_4=\chi_{123},$  $\chi_5=\cdots=\chi_N=\chi_{ij}$ for some $1\leq i<j\leq 3,$ and $\chi_i(x_i)=\chi_j(x_j).$

\end{itemize}
 \end{proposition}

\begin{proof}
The proof is divided into three cases. Note that all the rank $3$ admissible series of quasi-characters are contained in (1), so in (2-4) the ranks are assumed $>3.$

\emph{Case 1.} If there is a $\chi_j \ (j\geq 4)$ corresponding to $x_i$ for some $1\leq i\leq 3,$ then by (5.10) it follows that $\chi_j(x_i)=-\chi_i(x_i)$ and that $\chi_j(x_k)=\chi_i(x_k)$ for $1\leq k\neq i\leq 3.$

If there is another $\chi_k \ (k\geq 4)$ corresponding to $x_l$ for some $1\leq l\leq 3,$ then we have $l\neq i$ since otherwise $\chi_k(x_i)=-\chi_i(x_i)=-\chi_j(x_i),$ which is absurd. We also claim that there is no $\chi_r \ (r \ge 4)$ corresponding to $x_sx_t$ for some $1\leq s<t\leq 3.$ Otherwise, one of $\{i,l\},$ say $i,$ should be $s \ \mathrm{or}\ t.$ Then by (4.5) we have $\chi_r(x_i)\chi_i(x_sx_t)=1=\chi_r(x_i)\chi_j(x_sx_t),$ which clearly contradicts with the previous fact $\chi_i(x_sx_t)=-\chi_j(x_sx_t).$ Similarly, there is no quasi-character corresponding to $x_1x_2x_3.$ Therefore, all the possible admissible series of quasi-characters discussed above fall into (1) of our list. On the other hand, the admissibility of the series in (1) can be easily verified.

Next we assume other than $\chi_j,$ there is no other $\chi_k \ (k \geq 4)$ corresponding to $x_l$ for any $1\leq l\leq 3.$ Suppose some $\chi_r \ (r \ge 4)$ is corresponding to $x_sx_t$ for certain $1\leq s<t\leq 3.$ By the above discussion, we have $s \neq i \neq t.$ The condition $1 \ne \chi_r(x_sx_t)$ together with (5.6-5.10) implies that $\chi_s(x_s)=\chi_t(x_t).$ Then such series fall into (2) of the list. As for their admissibility, we have by direct computation that $\chi_r(x_i)\chi_j(x_sx_t)=1,$ and that $\chi_r(x_sx_t)=\chi_s(x_s)\chi_t(x_t)=-1$ since $\chi_s(x_s)=\chi_t(x_t)=\pm \mathrm{i}.$

\emph{Case 2.} Assume that there is no $\chi_j \ (j \geq 4)$ corresponding to $x_i$ for $1\leq i\leq 3$ or to $x_1x_2x_3.$

Since $N\geq 4,$ so there is some $\chi_m$ corresponding to $x_kx_l$ for certain $1\leq k<l\leq 3.$ We claim that there is no other $\chi_n \ (n \ge 4)$ corresponding to
$x_sx_t$ with $x_kx_l\neq x_sx_t.$ Otherwise, we have $k=s, l\neq t$ or $k\neq s, l=t.$ In either case, one can show that $\chi_m(x_n)\chi_n(x_m)=\chi_m(x_sx_t)\chi_n(x_kx_l)\neq 1$ which is a contradiction.  Thus, all the possible admissible series discussed previously fall into (3) of the list.

\emph{Case 3.} Assume that there is no $\chi_j \ (j \geq 4)$ corresponding to $x_i$ for any $1\leq i\leq 3,$ but there is some $\chi_k \ (k \ge 4)$ corresponding to $x_1x_2x_3.$

In this situation, if there is some $\chi_m$ corresponding to $x_kx_l$ for certain $1\leq k<l\leq 3,$ then by the preceding discussion we know that $\chi_m(x_kx_l) \neq 1$ implies $\chi_k(x_k)=\chi_l(x_l),$ and that there is no $\chi_n \ (n \ge 4)$ corresponding to $x_sx_t$ with $x_kx_l\neq x_sx_t.$ In addition, there is no other $\chi_r \ (r \ge 4)$ corresponding to $x_1x_2x_3$ by a simple checking of the admissibility. Then the admissible series discussed in this case fall into (4) of our list. As for their admissibility, one only needs to do the easy computation $\chi_m(x_4)\chi_4(x_m)=\chi_m(x_1x_2x_3)\chi_4(x_kx_l)=1.$
\end{proof}

\begin{corollary}
The following list provides an explicit presentation of the quasi-quantum linear spaces over $\Z_2 \times \Z_2 \times \Z_2$ associated to the list of admissible series of quasi-characters in Proposition 5.6:
\begin{itemize}
\item[(1)] $$x_mX_n=\chi_n(x_m)X_nx_m,\ \ 1\leq m\leq 3, \ 1\leq n\leq N;$$
$$X_mX_n=\chi_n(x_m)X_mX_n, \ \ 1\leq m\neq n\leq N;$$
$$X_m^{\overrightarrow{4}}=0, \ \ 1\leq m\leq N.$$
The coproduct is determined by $\bigtriangleup(g)=g\otimes g$ for $g\in G,$ $\bigtriangleup(X_m)=X_m\otimes 1+x_m\otimes X_m$ for $1\leq m\leq 3$ and
$\bigtriangleup(X_n)=X_n\otimes 1+x_l\otimes X_n$ for $4\leq n\leq N,$ here $\chi_n=\chi'_l$ for some $1\leq l\leq 3.$
\item[(2)] $$x_mX_n=\chi_n(x_m)X_nx_m,\ \ 1\leq m\leq 3, \ 1\leq n\leq N;$$
$$X_mX_n=\chi_n(x_m)X_mX_n, \ \ 1\leq m\neq n\leq N;$$
$$X_m^{\overrightarrow{4}}=0, \ 1\leq m\leq 4; \ \ X_n^{\overrightarrow{2}}=0, \ 5\leq n\leq N.$$
The coproduct is determined by $\bigtriangleup(g)=g\otimes g$ for $g\in G,$ $\bigtriangleup(X_m)=X_m\otimes 1+x_m\otimes X_m$ for $1\leq m\leq 3,$
$\bigtriangleup(X_4)=X_4\otimes 1+x_k\otimes X_4$ and $\bigtriangleup(X_n)=X_n\otimes 1+x_ix_j\otimes X_n$ for $5\leq n\leq N.$
\item[(3)] $$x_mX_n=\chi_n(x_m)X_ne_m,\ \ 1\leq m\leq 3, \ 1\leq n\leq N;$$
$$X_mX_n=\chi_n(x_m)X_mX_n, \ \ 1\leq m\neq n\leq N;$$
$$X_m^{\overrightarrow{4}}=0, \ 1\leq m\leq 3; \ \ X_n^{\overrightarrow{2}}=0, \ 4\leq n\leq N.$$
The coproduct is determined by $\bigtriangleup(g)=g\otimes g$ for $g\in G,$ $\bigtriangleup(X_m)=X_m\otimes 1+x_m\otimes X_m$ for $1\leq m\leq 3,$
and $\bigtriangleup(X_n)=X_n\otimes 1+x_ix_j\otimes X_n$ for $4\leq n\leq N.$
\item[(4)] $$x_mX_n=\chi_n(x_m)X_nx_m,\ \ 1\leq m\leq 3, \ 1\leq n\leq N;$$
$$X_mX_n=\chi_n(x_m)X_mX_n, \ \ 1\leq m\neq n\leq N;$$
$$X_m^{\overrightarrow{4}}=0, \ 1\leq m\leq 4; \ \ X_n^{\overrightarrow{2}}=0, \ 5\leq n\leq N.$$
The coproduct is determined by $\bigtriangleup(g)=g\otimes g$ for $g\in G,$ $\bigtriangleup(X_m)=X_m\otimes 1+x_m\otimes X_m$ for $1\leq m\leq 3,$
$\bigtriangleup(X_4)=X_4\otimes 1+x_1x_2x_3\otimes X_4$ and $\bigtriangleup(X_n)=X_n\otimes 1+x_ix_j\otimes X_n$ for $5\leq n\leq N.$
\end{itemize}
\end{corollary}

\begin{remark}
 Our method of computation may be applied to give a complete classification of finite quasi-quantum linear spaces, up to twist equivalence, over $\Z_2 \times \Z_2 \times \Z_2$ with nontrivial associators. As for a coboundary associator, the classification is reduced by gauge transformation to that of finite quantum linear spaces over $\Z_2 \times \Z_2 \times \Z_2$ which was contained in \cite{as1}. It is worth to note that, via a gauge transform of those quantum linear spaces over $\Z_2 \times \Z_2 \times \Z_2$ by the Albuquerque-Majid cochain related to the octonions, one naturally obtains all the finite octonionic quasi-quantum linear spaces!
\end{remark}

\vskip 5pt

\noindent {\bf Acknowledgement:} \quad
The authors are very grateful to the anonymous referee for the valuable comments and suggestions which highly improved the exposition.

\end{document}